\documentclass{amsart}
\usepackage{amsmath}
\usepackage{amssymb}
\usepackage{amsfonts}
\usepackage{amscd}
\usepackage{tikz}
\usepackage{bbm}

\usetikzlibrary{patterns} 
\usetikzlibrary{positioning}
\usetikzlibrary{calc}
\usetikzlibrary{arrows,shapes,backgrounds}
\usetikzlibrary{intersections}

\newtheorem{theorem}{Theorem}[section]

\numberwithin{theorem}{section}
\newtheorem{lemma}[theorem]{Lemma}
\newtheorem{corollary}[theorem]{Corollary}
\newtheorem{proposition}[theorem]{Proposition}
\newtheorem{definition}[theorem]{Definition}

\newtheorem{remark}[theorem]{Remark}

\numberwithin{equation}{section}

\begin{document}
\title[Reiterated homogenization for elliptic operators in Orlicz setting] %for nonlinear degenerate Elliptic Operators
%with Nonstandard Growth\  \ \ \ ]
{\ Reiterated Homogenization of nonlinear degenerate
elliptic operators with nonstandard growth\  \ \ }\
\author{Joel Fotso Tachago$^{\ddagger }$}
\curraddr{$^{\ddagger }$University of Bamenda, Higher Teachers Trainning
College, Department of Mathematics, P.O. Box 39, Bambili, Cameroon}
\email{fotsotachago@yahoo.fr}
\author{Hubert Nnang$^{\dagger }$}
\curraddr{$^{\dagger }$University of Yaounde I, \'{E}cole Normale Sup\'{e}%
rieure de Yaound\'{e}, P.O. Box 47 Yaounde, Cameroon.}
\email{hnnang@uy1.uninet.cm}
\author{Elvira Zappale$^{\intercal }$}
\curraddr{Dipartimento di Scienze di Base e Applicate per l'Ingegneria, Sapienza - Universit\`{a} 
di Roma, Via Antonio Scarpa, 16, Roma 10161, Italy }
\email{elvira.zappale@uniroma1.it}
%\date{June, 2020}

\begin{abstract}
It is shown by means of reiterated two-scale convergence in the Sobolev-Orlicz setting, that the
sequence of solutions of a class of highly oscillatory  problems involving
nonlinear elliptic operators with nonstandard growth, converges
to  a  solution of a suitable homogeneous nonlinear elliptic equation associated to an operator with nonstandard growth.
\end{abstract}

\maketitle
\section{Introduction}\label{sec1}
%\color{black}

\bigskip 

\noindent We are interested in the limiting behaviour (as $0<\varepsilon \rightarrow 0$) of the sequence of solutions of the problems%
\begin{equation}\label{1.1}
-{\rm div}\left[ a\left( \frac{x}{\varepsilon },\frac{x}{\varepsilon ^{2}}%
,u_{\varepsilon },Du_{\varepsilon }\right) \right] =f\text{ \ in }\Omega
,u_{\varepsilon }\in W_{0}^{1}L^{\Phi }\left( \Omega
\right) ,\text{ }
\end{equation}%
with $\Omega $ a regular bounded open set in $
\mathbb{R}^{d},d\geq 2,$ $D$ and ${\rm div}$ denoting gradient and divergence operators, respectively,
$f\in L^{d}\left( \Omega \right) \cap L^{\widetilde{\Phi }}\left( \Omega
\right) $, $a:=(a_i)_{1\leq i\leq d}:
\mathbb{R}
^{d}\times 
\mathbb{R}
^{d}\times 
\mathbb{R}
\times 
\mathbb{R}^{d}\rightarrow 
\mathbb{R}^{d}$ satisfying the following conditions:

$\left( H_{1}\right)$ For all $\left( \zeta ,\lambda \right) \in 
\mathbb{R}
\times 
\mathbb{R}
^{d},$ the function $\left( y,z\right) \longrightarrow a\left( y,z,\zeta
,\lambda \right) $ from $
\mathbb{R}
^{d}\times 
\mathbb{R}
^{d}$ into $
\mathbb{R}^{d}$ is of Caratheodory type, that is:

$\left( i\right) $ For each $z\in 
\mathbb{R}^{d},$ the function $y\longrightarrow a\left( y,z,\zeta ,\lambda \right) $
is measurable from $%
\mathbb{R}^{d}$ to $
\mathbb{R}^{d}$

$\left( ii\right) $ For almost all $y\in 
\mathbb{R}
^{d},$ the function $z\longrightarrow a\left( y,z,\zeta ,\lambda \right) $ is
continuous from $\mathbb{R}^{d}$ to $\mathbb{R}
^{d}$ with $a\left( \cdot, \cdot,0,\omega \right) \in L^{\infty }\left( 
\mathbb{R}_{y}^{d}\times 
\mathbb{R}
_{z}^{d}\right) ,\omega $ being the origin in $
\mathbb{R}^{d}.$

\noindent $\left( H_{2}\right) $ There are $N-$functions $\Phi ,\Psi :\left[ 0,+\infty
\right[ \rightarrow \left[ 0,+\infty \right[ ,\Phi ,\Psi $ being twice
continuously differentiable with%
\begin{equation}
1<\rho _{0}\leq \frac{t\psi \left( t\right) }{\Psi \left( t\right) }\leq
\rho _{1}\leq \frac{t\phi \left( t\right) }{\Phi \left( t\right) }\leq \rho
_{2}\text{ for all }t>0,  \label{1.2}
\end{equation}%
where \footnote{Recall that, as observed in \cite{MR} and \cite{Cle}, \eqref{1.2} guarantee that $\Phi, \Psi$ and their conjugates verify $\Delta_2$ condition} $\rho _{0},\rho _{1},\rho _{2}$ are constants and $\Phi ,\Psi $
are odd, increasing homeomorphisms from $
\mathbb{R}
$ to $
\mathbb{R}
$ such that $\Phi \left( t\right) =\int_{0}^{t}\phi \left( s\right) ds$ and $%
\Psi \left( t\right) =\int_{0}^{t}\psi \left( s\right) ds\left( t\geq
0\right) $. Moreover, there exist $c_{1},c_{3}>\frac{1}{2}$ and $c_{2},c_{4}>0$, $%
\Phi $ dominates $\Psi $ globally (in symbols $\Phi \prec \Psi$) and 
\begin{equation}
\left\vert a\left( y,z,\zeta ,\lambda \right) -a\left( y,z,\zeta',\lambda'\right) \right\vert \leq c_{1}\widetilde{\Psi }%
^{-1}\left( \Phi \left( c_{2}\left\vert \zeta -\zeta ^{\prime }\right\vert
\right) \right)+c_{3}\widetilde{\Phi }^{-1}\left( \Phi \left(
c_{4}\left\vert \lambda -\lambda \right\vert \right) \right)  \label{1.3}
\end{equation}%
for a.e. $y\in
\mathbb{R}
^{d}$ and for all $\left( z,\zeta ,\lambda \right) \in \mathbb{R}
^{d}\times 
\mathbb{R}
\times
\mathbb{R}
^{d},$ where $\widetilde{\Phi }\left( t\right) =\int_{0}^{t}\phi ^{-1}\left(
s\right) ds$ and $\widetilde{\Psi }\left( t\right) =\int_{0}^{t}\psi
^{-1}\left( s\right) ds\left( t\geq 0\right) $ are the complementary $N$-functions of $\Phi$ and $\Psi$, respectively, (see Section \ref{notations} for Orlicz-Sobolev spaces and \cite{Cle} for the adopted assumptions in the context of PDEs, among a wide literature on the subject).

\noindent $\left( H_{3}\right) $ There exists a continuous monotone decreasing mapping \ 
$h:\left[ 0,+\infty \right[ \rightarrow \left[ 0,1\right[ ,$ with $\underset{%
t\geq 0}{\min }h\left( t\right) >0$ and unbounded anti-derivative such that
for any $\left( \zeta ,\lambda \right) \in 
\mathbb{R}
\times
\mathbb{R}^d,$
\begin{equation}
a\left( y,z,\zeta ,\lambda \right) \cdot \lambda \geq \widetilde{\Phi }%
^{-1}\left( \Phi \left( h\left( \left\vert \zeta \right\vert \right)
\right) \right) \cdot\Phi \left( \left\vert \lambda \right\vert \right) \text{
a.e. }\left( y,z\right) \text{ in }%
\mathbb{R}^{d}\times 
\mathbb{R}^{d}.  \label{1.4}
\end{equation}%
\noindent $\left( H_{4}\right) $ For all $\zeta  \in 
\mathbb{R}
$ and for all $\lambda, \lambda' \in
\mathbb{R}
^{d},$ %there exists $c_5 >0$
\begin{align*}\left( a\left( y,z,\zeta ,\lambda \right) -a\left( y,z,\zeta,\lambda'\right) \right) \cdot \left( \lambda -\lambda'\right)  >  0  \hbox{ for a.e. } (y,z)\in
		\mathbb{R}^{d}\times 
		\mathbb{R}^{d}.%\color{black} \hbox{ prima come Nnang 2014} c_{5}\Phi \left( \left\vert \lambda -\lambda' \right\vert
%\right) 
\end{align*}
\color{black}

\noindent $\left( H_{5}\right) $ The function $a$ is periodic in the first two variables, and satisfies a local continuity assumption in the first variable, i.e.  
\begin{itemize}
	\item[(i)] $a\left( y+k,z+k',\zeta ,\lambda \right)
=a\left( y,z,\zeta ,\lambda \right) $ for any $\left( k,k'\right)
\in\mathbb{Z}^{d}\times \mathbb Z^d, \left( z,\zeta ,\lambda \right) \in
\mathbb{R}
^{d}\times \mathbb R \times  \mathbb R^d$ and a.e. $y\in 
\mathbb{R}
^{d};$
\item[(ii)]
For each bounded set $\Lambda $ in $
\mathbb{R}
^{d}$ and $\eta >0,$ there exists $\rho >0$ such  
that,  
\begin{align} \hbox{ if }\left\vert \xi \right\vert \leq \rho 
\hbox{ then }\left\vert a\left(
y-\xi ,z,\zeta ,\lambda \right) -a\left( y,z,\zeta ,\lambda \right)
\right\vert \leq \eta, \label{3.84} 
\end{align}
for all 
$\left( z,\zeta ,\lambda \right) \in 
\mathbb{R}^d\times\mathbb{R}\times \mathbb{R}^{d}$ and almost all $y\in 
\Lambda.
$
 \end{itemize}

%\color{red} METTERE UN' IPOTESI COME LA $(A5)$ di Amaziane et al.
\color{black}

Indeed, we aim at extending \cite[Theorem 1.3]{All2} and \cite[Theorem 29]{LNW}, to the framework of Sobolev-Orlicz spaces, relying on the ad hoc notion of reiterated two-scale convergence in such spaces, obtained in \cite{FNZOpuscula} (cf. also \cite{fotso nnang 2012} and \cite{FTGNZ}). 
We also refer to \cite{NN, Nnang These, Nnang Orlicz 2014} for homogenization problems for PDEs in the Orlicz setting, to \cite{W2008, W12010, W32010, W22010} for homogenization of non-monotone operators in the Sobolev setting, to \cite{AAP, MMR, MRT} for homogenization problems in the variable exponent setting, among a wider literature and to \cite{nnang reit} for reiterated homogenization in general deterministic setting.

Indeed, under the above assumptions (which, in turn, rephrase into the multiscale periodic setting, the degenerate equation considered in \cite{Y}) and with the notation in section \ref{notations} and subsection \ref{sub32},  our main results read as follows:

\begin{theorem}\label{mainresult}
	Let \eqref{1.1} be the problem defined in Section \ref{sec1}, with $a$ and $f$ satisfying $(H_1)-(H_5)$. For each $\varepsilon
	>0$, let $u_{\varepsilon }$ be a solution of \eqref{1.1}. Then there exists a not relabeled subsequence  and $u:=\left( u_{0},u_{1},u_{2}\right)\in \mathbb{F}_{0}^{1,\Phi }:=W_{0}^{1}L^{\Phi}\left( \Omega
	\right) \times L^\Phi_{per}\left( \Omega ;W_{\#}^{1}L^{\Phi}\left( Y
	\right) \right) \times L^\Phi\left( \Omega ;L_{per}^{\Phi}\left(
	Y;W_{\#}^{1}L^{\Phi}\left( Z
	\right) \right) \right)$ such that
	\begin{equation}
		u_{\varepsilon }\rightharpoonup u_{0}\text{ in }W_{0}^{1}L^{\Phi }\left( \Omega
		\right) -\text{weakly},  \label{3.58}
	\end{equation}%
	\begin{align}
		D_{x_{i}}u_{\varepsilon }\rightharpoonup
		D_{x_{i}}u_{0}+D_{y_{i}}u_{1}+D_{z_{i}}u_{2} \nonumber \\
		\text{ weakly reiteratively two-scale in }L^{\Phi }(\Omega), 1\leq i\leq N, \label{3.59}
	\end{align}
	and $u$ %:=\left( u_{0},u_{1},u_{2}\right)$ %\in %\mathbb{F}_{0}^{1,\Phi %}:=W_{0}^{1}L^{\Phi}\left( \Omega
	%\right) \times L^\Phi_{\color{red}per}\left( %\Omega ;W_{\#}^{1}L^{\Phi}\left( Y
%	\right) \right) \times L^\Phi\left( \Omega %;L_{per}^{\Phi}\left(
	%Y;W_{\#}^{1}L^{\Phi}\left( Z
	%\right) \right) \right)$ % is uniquely %defined by
	solves the problem
	\begin{equation}
		\left\{ 
		\begin{tabular}{l}
			$\int_{\Omega }\int_{Y}\int_{Z}a\left(
			y, z,u_{0},Du_{0}+D_{y}u_{1}+D_{z}u_{2}\right)\cdot \left(
			Dv_{0}+D_{y}v_{1}+D_{z}v_{2}\right) dxdydz$ \\
			\\ 
			$=\int_{\Omega }fv_{0}dx,$ for all $v=\left( v_{0},v_{1},v_{2}\right) \in 
			\mathbb{F}_{0}^{1,\Phi }$. \\ 
		\end{tabular}%
		\right. . \label{3.60}
	\end{equation}
\end{theorem}

 Furthermore, in order to get uniqueness of the  solutions in \eqref{1.1} and \eqref{3.60},  in the same spirit of \cite[(2.3.40)]{Pankov1997} (see also \cite{P}) one can assume that there exists $c_5>0$ such that
 
 \noindent $\left( H_{6}\right) $ for all $\zeta, \zeta' \in 
 \mathbb{R}
 $ and for all $\lambda, \lambda' \in
 \mathbb{R}
 ^{d},$
 \begin{align*}\left( a\left( y,z,\zeta ,\lambda \right) -a\left( y,z,\zeta',\lambda'\right) \right) \cdot \left( \lambda -\lambda'\right)  > c_5\Phi \left( \left\vert \lambda -\lambda' \right\vert
 	\right) 
 \end{align*} a.e in $\left( y,z\right) $ in $
 \mathbb{R}^{d}\times 
 \mathbb{R}^{d}$, see Remark \ref{remuniq} below.
 
\begin{theorem}\label{maincor} 
For every $\varepsilon >0$, let \eqref{1.1} be  such that $a$ and $f$ satisfy $(H_1)-(H_6)$. 
%For each $\varepsilon
%>0$, let $u_{\varepsilon }$ be a solution of \eqref{1.1}. Then
%\begin{equation}
%	u_{\varepsilon }\rightharpoonup u_{0}\text{ in }W_{0}^{1}L^{\Phi }\left( \Omega
%	\right) -\text{weakly}  \label{3.58}
%\end{equation}%
%and%
%	\begin{align}
	%		D_{x_{i}}u_{\varepsilon }\rightharpoonup
	%		D_{x_{i}}u_{0}+D_{y_{i}}u_{1}+D_{z_{i}}u_{2} \nonumber \\
	%		\text{ weakly reiteratively two-scale in }L_{per}^{\Phi }(\Omega), 1\leq i\leq N, \label{3.59}
	%	\end{align}
%	where the function $u:=\left( u_{0},u_{1},u_{2}\right)\in \mathbb{F}_{0}^{1,\Phi } $ is uniquely defined by
Let $u_0 \in W_{0}^{1}L^{\Phi }(\Omega)$ be the solution  defined by means of \eqref{3.60}. Then, it is the unique solution of the
macroscopic homogenized problem 
\begin{equation}
	-{\rm div} q\left(u_0, Du_{0}\right) =f\text{ in }\Omega ,u_{0}\in W_{0}^{1}L^{\Phi
	}\left( \Omega \right),  \label{3.72}
\end{equation}
where $q$ is defined as follows.
For $(r,\xi) \in \mathbb R\times \mathbb R^d$
\begin{align}
	\label{q}
	q\left( r,\xi \right) =\int_{Y}h\left(y,r, \xi +D_{y}\pi
	_{1}\left( r,\xi \right) \right) dy.
\end{align}  
where, 
 for a.e. $y \in Y$, and for any $(r, \xi) \in \mathbb R \times \mathbb R^d$,  
\begin{equation}\label{h}h\left(y,r, \xi \right) :=\int_{Z}a_{i}\left( y,z,r ,\xi
	+D_{z}\pi _{2}\left(y,r, \xi \right) \right) dz,\end{equation}

where for a.e.  $y \in Y$, and every $(r,\xi) \in \mathbb R \times \mathbb R^d$,$\pi_2(y, r,\xi)$,  is the solution of the following variational cell problem:
\begin{equation}
	\left\{ 
	\begin{tabular}{l}
		$\hbox{find } \pi _{2}\left(y,r, \xi \right) \in W_{\#}^{1}L^\Phi\left( Z\right) $ \hbox{such that} \\ 
		$\int_{Z}a\left( y,z,r,\xi +D_{z}\pi _{2}\left( y,r,\xi \right) \right)
		\cdot D_{z}\theta dz=0$ for all $\theta \in W_{\#}^{1}L^\Phi\left( Z\right) $%
	\end{tabular}%
	\right.  \label{3.67}
\end{equation}%and for any $(r,\xi) \in \mathbb R\times
 and $\pi _{1}\left(r, \xi \right) \in
W_{\#}^{1}L^\Phi\left( Y
\right) $ is the unique solution of the variational problem 
\begin{equation}
	\left\{ 
	\begin{tabular}{l}
		\hbox{ find} $\pi _{1}\left(r, \xi \right) \in W_{\#}^{1}L^{\Phi}\left( Y
		\right) $ \hbox{ such that } \\ 
		$\int_{Y}h\left( r, \xi +D_{y}\pi _{1}\left(r, \xi \right) \right) \cdot D_{y}\theta
		dy=0$ for all $\theta \in W_{\#}^{1}L^{\Phi}\left( Y
		\right). $%
	\end{tabular}%
	\right.  \label{3.69b}
\end{equation}

\end{theorem}

%\color{red}{SPOSTARE IN SECTION 4}
%\begin{theorem}\label{mainresult}
%%Let \eqref{1.1} be the problem defined in Section %\ref{sec1}, with $a_i$ and $f$ satisfying %$(H_1)-(H_4)$. 
%For each $\varepsilon
%>0$, let $u_{\varepsilon }$ be a solution of \eqref{1.1}. Then
%\begin{align*}
%u_{\varepsilon }\rightharpoonup u_{0}\text{ in }W_{0}^{1}L^{\Phi }\left( \Omega
%\right) -\text{weakly }, \hbox{ and } %\label{3.58}
%\\
%%\begin{align*}
%D_{x_{i}}u_{\varepsilon }\rightharpoonup
%D_{x_{i}}u_{0}+D_{y_{i}}u_{1}+D_{z_{i}}u_{2} \nonumber \\
%\text{ weakly reiteratively two-scale in }L_{per}^{\Phi }(\Omega), 1\leq i\leq N, %\label{3.59}
%\end{align*}
%where the function $u:=\left( u_{0},u_{1},u_{2}\right)\in \mathbb{F}_{0}^{1,\Phi } $ is uniquely defined by
%\begin{equation*}
%\left\{ 
%\begin{tabular}{l}
%$\int_{\Omega }\int_{Y}\int_{Z}a_{i}\left(
%y, z,u_{0},Du_{0}+D_{y}u_{1}+D_{z}u_{2}\right)\cdot \left(
%Dv_{0}+D_{y}v_{1}+D_{z}v_{2}\right) dxdydz$ \\
%\\ 
%$=\int_{\Omega }fv_{0}dx,$ for all $v=\left( v_{0},v_{1},v_{2}\right) \in 
%\mathbb{F}_{0}^{1,\Phi }$. \\ 
%\end{tabular}%
%\right.  %\label{3.60}
%\end{equation*}
%\end{theorem}
\color{black}

% \color{magenta} Insert Youssfi e mgari Alvino etc \color{black}

The paper is organized as follows: section 2 deals with some preliminaries on Orlicz-Sobolev spaces, reiterated two-scale convergence, compactness results in the considered functions spaces, and other preliminaries while section 3 focuses on the detection of the asymptotic behaviour  of solutions of problems with highly oscillating coefficients, in particular \eqref{1.1} while section 4 contains the proof of our main result. Finally in the Appendix, for the reader's convenience, we justify the well-posedness of \eqref{1.1} under our set of assumptions.  

\section{Notation and preliminary results}\label{notations}

In what follows $X$ and
$V$ denote a locally compact space and a Banach space, respectively, and
$\mathcal C(X; V)$ stands for the space of continuous functions from $X$ into $V$, and
$\mathcal C_b(X; V)$ stands for those functions in $\mathcal C(X; V)$ that are bounded. The space $\mathcal C_b(X; V)$ is enodowed with the supremum norm $\|u\|_{\infty} = \sup_{x\in X}
\|u(x)\|$, where
$\|\cdot\|$ denotes the norm in $V$, (in particular, given an open set $A\subset \mathbb R^d$ by $\mathcal C_b(A)$ we denote the space of real valued continuous and bounded functions defined in $A$). Likewise the spaces $L^p(X; V)$ and $L^p_{\rm loc}(X; V)$
($X$ provided with a positive Radon measure) are denoted by $L^p(X)$ and
$L^p_{\rm loc}(X)$, respectively, when $V = \mathbb R$ (we refer to \cite{FLbook} for integration theory).

In the sequel we denote by $Y$ and $Z$ two identical copies of the cube $]-1/2,1/2[^d$.  

In order to enlighten the space variable under consideration we will adopt the notation $\mathbb R^d_x, \mathbb R^d_y$, or $\mathbb R^d_z$ to indicate where $x,y $ or $z$ belong to.

The family of open subsets in $\mathbb R^d_x$ will be denoted by $\mathcal A(\mathbb R^d_x)$.

For any subset $E$ of $\mathbb R^m$, $m \in \mathbb N$, by $\overline E$, we denote its closure in the relative topology.

For every $x \in \mathbb R^d$ we denote by $[x]$ its integer part, namely the vector in $\mathbb Z^d$, which has as components the integer parts of the components of $x$.

By $\mathcal L^d$ we denote the Lebesgue measure in $\mathbb R^d$.

\subsection{Orlicz-Sobolev spaces}\label{O-Sspaces}

\bigskip Let $B:\left[ 0,+\infty \right[ \rightarrow \left[ 0,+\infty \right[
$ be an ${\rm N}-$function (see \cite{ada}), i.e., $B$ is continuous, convex, with $%
B\left( t\right) >0$ for $t>0,\frac{B\left( t\right) }{t}\rightarrow 0$ as $t\rightarrow 0,$ and $\frac{B\left( t\right) }{t}\rightarrow \infty $ as $%
t\rightarrow \infty .$
Equivalently, $B$ is of the form $B\left( t\right)
=\int_{0}^{t}b\left( \tau \right) d\tau ,$ where $b:\left[ 0,+\infty \right[
\rightarrow \left[ 0,+\infty \right[ $ is non decreasing, right continuous,
with $b\left( 0\right) =0,b\left( t\right) >0$ if $t>0$ and $b\left(
t\right) \rightarrow +\infty $ if $t\rightarrow +\infty .$ 
                                                                                                                                                    
We denote by $\widetilde{B},$ the complementary ${\rm N}-$function of $B$ defined by $$\widetilde{B}(t)=\sup_{s\geq 0}\left\{ st-B\left( s\right) \right\}, \,t\geq 0.
$$ It follows
that 
\begin{equation}\nonumber 
\frac{tb(t)}{B(t)} \geq 1
\;\;(\hbox{or }> \hbox{if }b\hbox{ is strictly increasing}),
\end{equation}
\begin{equation}
\nonumber \widetilde{B}( b(t) )\leq
tb( t) \leq B( 2t) \hbox{ for all }t>0.
\end{equation}
An ${\rm N}-$function $B$ is of class $\triangle _{2}$ near $\infty$ (denoted $B\in \triangle
_{2}$) if there are $\alpha >0$ and $t_{0}\geq 0$ such that 
\begin{equation}\label{Delta2}B\left(
2t\right) \leq \alpha B\left( t\right) 
\end{equation} for all $t\geq t_{0}$.

\noindent An  $N$- function $B$ is of class $\Delta'$ if there exists $C>0$ such that
$B(ts) \leq C B(t)B(s)$, for every $s,t\geq 0$.
\color{black}
\noindent In what
follows every $N$- function $B$  and its conjugate $\widetilde{B}$ 
satisfy the $\triangle_2$ condition and $c$ refers to a constant.

It is also worth recalling that given two $N$-functions
$B$ and $C$, $B$ dominates $C$ (denoted as $C \prec B$) if there is $k>0$ such that $C(t) \leq B(k t)$ for all $t \geq 0$. Hence, it follows that $C \prec B$ if and only if 
${\widetilde B} \prec {\widetilde C}$.

Let $\Omega $ be a
bounded open set in $\mathbb R^d$. The Orlicz space
\begin{equation*}L^{B}\left(
\Omega \right) =\left\{ u:\Omega \rightarrow 
\mathbb R \hbox{ measurable},\lim_{\delta \to 0^+} \int_{\Omega
}B\left( \delta \left\vert u\left( x\right) \right\vert \right) dx=0\right\} 
\end{equation*}
is a Banach space with respect to the Luxemburg norm: \begin{equation*}\left\Vert u\right\Vert
_{B,\Omega }:=\inf \left\{ k>0:\int_{\Omega }B\left( \frac{\left\vert u\left(
	x\right) \right\vert }{k}\right) dx\leq 1\right\} <+\infty .\end{equation*}It follows
that: $\mathcal{D}(\Omega)$ is dense in $L^{B}\left(\Omega
\right)$, $L^{B}\left(\Omega \right)$ is separable and reflexive, the dual
of $L^{B}\left( \Omega \right) $ is identified with $L^{\widetilde{B}}\left(
\Omega \right),$ and the norm on $L^{\widetilde{B}}\left( \Omega \right) $
is equivalent to $\left\Vert \cdot\right\Vert _{\widetilde{B},\Omega }.$
We will denote the norm of elements in $L^{B}\left( \Omega \right)$, both by $\|\cdot\|_{L^{B}\left( \Omega \right)}$ and with $\|\cdot\|_{B, \Omega}$, the latter symbol being useful when we want to emphasize the domain $\Omega$.

Furthermore, it is also convenient to recall that:
\begin{itemize} 
	\item[(i)] $
	\left\vert \int_{\Omega }u\left( x\right) v\left( x\right) dx\right\vert
	\leq 2\left\Vert u\right\Vert _{B,\Omega }\left\Vert v\right\Vert _{%
		\widetilde{B},\Omega }$ for $u\in L^{B}\left( \Omega \right) $ and $v\in L^{%
		\widetilde{B}}\left( \Omega \right) $, 
	\item[(ii)] given $v\in L^{%
		\widetilde{B}}\left( \Omega \right)$, the linear functional $L_{v}$ on $
	L^{B}\left( \Omega \right) $ defined by $L_{v}\left( u\right)$ $:=\int_{\Omega
	}u\left( x\right) v\left( x\right) dx,$ $\left( u\in L^{B}\left( \Omega \right)
	\right) $ belongs to the dual $\left[ L^{B}\left( \Omega \right) \right]
	^{\prime }=L^{\widetilde{B}}\left( \Omega \right) $ with $\left\Vert
	v\right\Vert _{\widetilde{B},\Omega }\leq \left\Vert L_{v}\right\Vert _{\left[ L^{B}\left( \Omega \right) \right] ^{\prime }}\leq 2\left\Vert
	v\right\Vert _{\widetilde{B},\Omega }$, 
	\item[(iii)]  the property $\lim_{t \to +\infty} \frac{B\left( t\right) }{t}=+\infty $
	implies $L^{B}\left( \Omega \right) \subset L^{1}\left( \Omega \right)
	\subset L_{loc}^{1}\left( \Omega \right) \subset \mathcal{D}^{\prime }\left(
	\Omega \right),$ each embedding being continuous.
\end{itemize}

Given any $d\in \mathbb N$, when $u:\Omega \to \mathbb R^d$, such that each component $u^i$, of $u$, lies in $L^B(\Omega)$,  
we will denote the norm of $u$ with the symbol $\|u\|_{L^B(\Omega)^{d}}:=\sum_{i=1}^d \|u^i\|_{B,\Omega}$.

Analogously one can define the Orlicz-Sobolev function space as follows: 

\noindent$%
W^{1}L^{B}\left( \Omega \right) =\left\{ u\in L^{B}\left( \Omega \right) :%
\frac{\partial u}{\partial x_{i}}\in L^{B}\left( \Omega \right),1\leq i\leq
d\right\},$ where derivatives are taken in the distributional sense on $%
\Omega.$ Endowed with the norm $\left\Vert u\right\Vert _{W^{1}L^{B}\left(
	\Omega \right) }=\left\Vert u\right\Vert _{B,\Omega }+\sum_{i=1}^{d}$ $%
\left\Vert \frac{\partial u}{\partial x_{i}}\right\Vert _{B,\Omega },u\in
W^{1}L^{B}\left( \Omega \right) ,$\ \ $W^{1}L^{B}\left( \Omega \right) $ is
a reflexive Banach space. We denote by $W_{0}^{1}L^{B}\left( \Omega \right)
, $ the closure of $\ \mathcal{D}\left( \Omega \right) $\ in $%
W^{1}L^{B}\left( \Omega \right) $ and the semi-norm $u\rightarrow \left\Vert
u\right\Vert _{W_{0}^{1}L^{B}\left( \Omega \right) }=\left\Vert
Du\right\Vert _{B,\Omega }=\sum_{i=1}^{d}$ $\left\Vert \frac{\partial u}{%
	\partial x_{i}}\right\Vert _{B,\Omega }$ is a norm on $W_{0}^{1}L^{B}\left(
\Omega \right) $ equivalent to $\left\Vert \cdot \right\Vert _{W^{1}L^{B}\left(
	\Omega \right) }.$

By $W_{\#}^{1}L^{B}\left( Y\right)$, we denote the space of functions $u \in W^1L^B(Y)$ such that $\int_Y u(y)d y=0$.  It is endowed with the gradient norm.

Given a function space $S$ defined in $Y$, $Z$ or $Y\times Z$, the subscript $_{per}$ stands for periodic, i.e. $S_{per}$ means that its elements are periodic in $Y$, $Z$ or $Y\times Z$, as it will be clear from the context. In particular $\mathcal C_{per}(Y\times Z)$ denotes the space of periodic functions in $\mathcal C(\mathbb R^d_y\times \mathbb R^d_z)$, i.e. that verify $w(y + k, z + h) = w(y, z)$ for $(y, z) \in \mathbb R^d \times \mathbb R^d$
and $(k, h) \in \mathbb Z^d \times \mathbb Z^d$. $\mathcal C^\infty_{per}(Y\times Z)=\mathcal C_{per}(Y\times Z)\cap \mathcal C^\infty(\mathbb R^d_y\times \mathbb R^d_z)$, and $L^B
_{per}(Y \times Z)$ is the space of
$Y \times Z$ -periodic functions in $L^B_{loc}(\mathbb R^d_y
\times \mathbb R^d_z)$.
In our subsequent analysis we denote by $L^B(\Omega; L^B_{per}(Y))$ and $L^B(\Omega;L^B_{per}(Y \times Z))$ the spaces of functions in $L^B_{\rm loc}(\Omega \times Y)$ and $L^B_{\rm loc}(\Omega \times \mathbb R^d_y \times \mathbb R^d_z)$, respectively which are $Y$ and  $Y\times Z$ periodic
for a.e. $x \in \Omega$, respectively and whose Luxemburg  norm is finite in $\Omega \times K$, with $K$ being any compact set in  $Y$ and $Y \times Z$, respectively.

 In formulas
\begin{align}\nonumber
	L^{B}\left( \Omega; L^B_{per}( Y) \right) :=\Big\{ u\in
	L_{loc}^{B}\left( \Omega \times 
	\mathbb{R}
	_{y}^d\right) : u\left( x,\cdot\right) \in
	L_{per}^{B}\left( Y\right)  \\ 
	\left. \hbox{for a.e. }x\in \Omega,\text{ and }\iint_{\Omega \times Y}B\left( \left\vert u\left(
	x,y\right) \right\vert \right) dxdy<\infty \right\},\label{LBperY}
\end{align}
 \begin{align}\nonumber
	L^{B}\left( \Omega;L^B_{per} (Y\times Z)\right) :=\Big\{ u\in
	L_{loc}^{B}\left( \Omega \times 
	\mathbb{R}
	_{y}^d\times \mathbb{R}
	_{z}^d\right) : u\left( x,\cdot,\cdot\right) \in
	L_{per}^{B}\left( Y\times Z\right)  \\ 
	\left. \hbox{for a.e. }x\in \Omega,\text{ and }\iiint_{\Omega \times Y\times Z}B\left( \left\vert u\left(
	x,y,z\right) \right\vert \right) dxdydz<\infty \right\},\label{LBper}
\end{align}
respectively.  We observe that, in view of \cite[Lemma 2.4]{fotso nnang 2012}, if $\tilde B$ satisfies $\Delta'$ condition, then the above spaces coincide with the standard Orlicz-Bochner spaces.
\color{black} These spaces play an important role in the definition of reiterated two-scale convergence in the Orlicz setting.

\subsection{Traces results}\label{traces}
This subsection is devoted to recall some results which are crucial for reiterated multiple scales convergence in the Orlicz setting. %We refer to \cite{fotso nnang 2012, FNZOpuscula, FTGNZ} for more details and proofs. 

While the definitions are natural for regular functions, several function spaces and related norms can be introduced to extend the concept of compositions to the multiscale, periodic setting.
% introduced together with proofs of functions spaces' properties. 
% Not all the results proved here will be explicitly recalled in the next subection or in the remainder of the paper, since, in the sequel, 
%On the other hand we will not present neither arguments which are very similar to the ones used to deal with standard two scale convergence in the Orlicz setting, nor those related to reiterated two-scale convergence in 
The notation is very similar to \cite[Sections 2 and 4]{nnang reit}) and \cite[Section 2 and Appendix]{FNZOpuscula}, where also proofs dealing with the standard Sobolev setting can be found. 

Traces of the form $u^\varepsilon(x):=u\left( x,\frac{x}{\varepsilon 
},\frac{x}{\varepsilon ^{2}}\right) ,x\in \Omega,$ $\varepsilon >0$, when $u\in \mathcal{C}\left( \Omega \times 
\mathbb R_y^d\times
\mathbb R_z^d\right) $ are well known and, clearly the operator
%We define %$u^{\varepsilon }\left( x\right) $ by
%\begin{equation*}
%u^{\varepsilon }\left(
%x\right) :=u\left( x,\frac{x}{\varepsilon },\frac{x}{\varepsilon ^{2}}\right)
%\end{equation*}
%%,$where $u\left( x,\frac{x}{\varepsilon },\frac{x}{\varepsilon ^{2}}\right)
%%=u\left( x,y,z\right) ,$ $y=\frac{x}{\varepsilon },z=\frac{x}{\varepsilon
%%^{2}},x\in \Omega .$ 
%Obviously $u^{\varepsilon }\in \mathcal{C}\left( \Omega
%\right) .$ We define the trace operator 
of order $\varepsilon
>0,(t^{\varepsilon })$, defined by 
\begin{equation}
\label{traceoperator}t^{\varepsilon }:u \in \mathcal{C}\left( \Omega
\times 
\mathbb R
_{y}^d\times 
\mathbb R_z^d\right) \longrightarrow u^\varepsilon \in \mathcal{C}\left( \Omega \right),
\end{equation}
is linear and continuous.

%\begin{proof}[Proof]
%Let $K\subset \Omega $ a compact set and $\varepsilon >0,$as $x\rightarrow 
%\frac{x}{\varepsilon },x\rightarrow \frac{x}{\varepsilon ^{2}}$ are
%continuous from $%
%TCIMACRO{\U{211d} }%
%BeginExpansion
%\mathbb{R}
%EndExpansion
%^{N}\rightarrow 
%TCIMACRO{\U{211d} }%
%BeginExpansion
%\mathbb{R}^{N},\frac{K}{\varepsilon },\frac{K}{\varepsilon ^{2}}$are compact and $%
%K\times \frac{K}{\varepsilon }\times \frac{K}{\varepsilon ^{2}}$ is compact
%in $\Omega \times 
%TCIMACRO{\U{211d} }%
%BeginExpansion
%\mathbb{R}
%EndExpansion
%_{y}^{N}\times 
%TCIMACRO{\U{211d} }%
%BeginExpansion
%\mathbb{R}
%EndExpansion
%_{z}^{N}.$ Moreover $\underset{x\in K}{\sup }\left\{ \left\vert
%u^{\varepsilon }\left( x\right) \right\vert \right\} =\underset{x\in K}{\sup 
%}\left\{ \left\vert u\left( x,\frac{x}{\varepsilon },\frac{x}{\varepsilon
%^{2}}\right) \right\vert \right\} \leq \underset{\left( x,y,z\right) \in
%K\times \frac{K}{\varepsilon }\times \frac{K}{\varepsilon ^{2}}}{\sup }%
%\left\{ \left\vert u^{\varepsilon }\left( x,y,z\right) \right\vert \right\} $
%hence continuity follow.
%\end{proof}

Making use of the subscript $_b$ to denote bounded functions, the same definitions and properties hold true (since $\overline \Omega$ is compact), when
$u\in \mathcal{C}\left( \overline{\Omega };\mathcal C_b\left( 
\mathbb R_{y}^d\times 
\mathbb R_z^d\right) \right) \subset \mathcal{C}\left( \overline{\Omega };%
\mathcal{C}\left(
\mathbb R_y^d\times 
\mathbb R_z^d\right) \right) \widetilde{=}\mathcal{C}\left( \overline{\Omega }%
\times 
\mathbb R_y^d\times
\mathbb R_z^d\right) .$

\noindent Then, considering $\mathcal{C}\left( \overline{\Omega }%
;\mathcal C_b\left( 
\mathbb R_y^d\times 
\mathbb R_z^d\right) \right) $ as a subspace of $\mathcal{C}\left( \overline{%
	\Omega }\times 
\mathbb R_y^d\times \mathbb R_z^d\right) $,
%Since $\overline{\Omega }$\ \ is compact in $\mathbb R_{x}^{N},$ then 
$u^{\varepsilon }\in \mathcal C_b\left( \Omega \right) $ and, with an abuse of notation the operator
$t^\varepsilon$ can be interpreted from $\mathcal{C}\left( \overline{\Omega };%
\mathcal C_b\left(
\mathbb{R}_{y}^d\times 
\mathbb{R}
_{z}^d\right) \right) $ to $\mathcal C_b\left( \Omega \right)$
as linear and continuous. Moreover, it is easily seen that
\begin{equation}\label{estuepsi}
%Let $u\in \mathcal{C}\left( 
 \left\vert u^{\varepsilon }\left( x\right)
\right\vert =\left\vert u\left( x,\frac{x}{\varepsilon },\frac{x}{%
	\varepsilon ^{2}}\right) \right\vert \leq \left\Vert u(x)
\right\Vert_{\infty }
\end{equation}for every $x \in \Omega$.
By $u\in L^{B}(\Omega ;\mathcal C_b\left( \mathbb R _y^d\times 
\mathbb R_z^d\right)) $ we mean that the function $x\rightarrow \left\Vert
u\left( x\right) \right\Vert _{\infty },$ from $\Omega $ into $
\mathbb R 
$, belongs to $L^{B}\left( \Omega \right) $ and
\begin{equation*}
\left\Vert u\right\Vert _{L^{B}\left( \Omega ;\mathcal C_b\left(\mathbb R
	_{y}^d\times 
	\mathbb R_{z}^d\right) \right) }=\inf \left\{ k>0:\int_{\Omega }B\left( \frac{
	\left\Vert u\left( x\right) \right\Vert _{\infty }}{k}\right) dx\leq
1\right\} <+\infty .
\end{equation*}

 Recalling that ${\rm N}-$functions are non decreasing, from \eqref{estuepsi}, we deduce that: 
\begin{align*}
B\left( \frac{\left\vert u^{\varepsilon }\left( x\right) \right\vert }{k}%
\right) \leq B\left( \frac{\left\Vert u\left( x\right) \right\Vert _{\infty }%
}{k}\right) ,\hbox{ for all } k>0,\hbox{ for all } x\in \overline{\Omega },
\\
\int_{\Omega }B\left( \frac{\left\vert u^{\varepsilon }\left(
	x\right) \right\vert }{k}\right) dx\leq \int_{\Omega }B\left( \frac{%
	\left\Vert u\left( x\right) \right\Vert _{\infty }}{k}\right) dx,
\\
\hbox{ thus }
\int_{\Omega }B\left( \frac{\left\Vert u\left( x\right) \right\Vert _{\infty
}}{k}\right) dx\leq 1\Longrightarrow \int_{\Omega }B\left( \frac{\left\vert
	u^{\varepsilon }\left( x\right) \right\vert }{k}\right) dx\leq 1,
\end{align*}
hence
\begin{align}\label{tracebounds}
	\left\Vert u^{\varepsilon }\right\Vert _{L^{B}\left( \Omega \right) }\leq
\left\Vert u\right\Vert _{L^{B}\left( \Omega ;\mathcal C_b\left( 
	\mathbb{R}
	_{y}^d\times 
	\mathbb{R}
	_{z}^d\right) \right) }.\end{align}
Thus the trace operator $t^\varepsilon: u\rightarrow
u^{\varepsilon }$ from $\mathcal{C}\left( \overline{\Omega }; \mathcal C_b\left( \mathbb R _y^d\times 
\mathbb R_z^d\right)\right) $ into $L^{B}\left( \Omega \right) ,$ extends by
density and continuity to a unique operator from $L^{B}( \Omega ;\mathcal C_b\left( \mathbb R _y^d\times 
\mathbb R_z^d\right)) $, still denoted in the same way, which
verifies %\begin{align}
\eqref{tracebounds}
%\left\Vert u^{\varepsilon
%}\right\Vert _{L^{B}\left( \Omega \right) }\leq \left\Vert u\right\Vert
%_{L^{B}\left( \Omega ;\mathcal C_b\left( 
%	\mathbb{R}
%	_{y}^d\times 
%	\mathbb{R}
%	_{z}^d\right) \right) },  
%\end{align}
for all $u\in L^{B}\left( \Omega ;\mathcal C_b\left( \mathbb R _y^d\times 
\mathbb R_z^d\right) \right)$.
\noindent 
Referring to \cite[Section 2]{nnang reit} and to \cite[Section 2]{FNZOpuscula}, it can be ensured 
measurability for the trace of elements $u\in L^{\infty }\left( 
\mathbb{R}
_{y}^d;\mathcal C_b\left( 
\mathbb{R}
_{z}^d\right) \right) $ and $u\in \mathcal{C}\left( \overline{\Omega }%
;L^{\infty }\left( 
\mathbb{R}
_{y}^d;\mathcal C_b\left( 
\mathbb{R}
_{z}^d\right) \right) \right) $, which is of crucial importance to deal with reiterated to-scale convergence.
% we need to have good definition for the measurability of test functions, so we should ensure , but we omit these proofs, 

By $M:\mathcal C_{per}\left( Y\times Z\right) \rightarrow 
\mathbb R$ we denote the mean value operator (or equivalently `averaging operator') defined by
\begin{equation}
\label{M}
u\rightarrow M(u):=\iint_{Y\times Z}u\left( x,y\right) dxdy.
\end{equation}
It is easily seen that it is:
\begin{itemize}
	\item[(i)] nonnegative, i.e. $M\left( u\right) \geq 0\
	\hbox{ for all } u\in \mathcal{C}_{per}(Y\times Z) ,u\geq 0;$
	\item[(ii)] continuous on $\mathcal{C%
	}_{per}\left( Y\times Z\right) $ (for the sup norm); 
	\item[(iii)] such that
	$M\left( 1\right) =1$;
	\item[(iv)]  translation invariant.
\end{itemize}

%{\bf where do we use this?}
%It is a well known result that, given $\gamma %_{h,l}\left( y,z\right)
%=e^{2i\pi \left( h\cdot y+l\cdot z\right) }$ with %$,$ $h\cdot y+l\cdot z=%
%\sum_{j=1}^{N}h_{j}y_{j}+l_{j}z_{j},h=\left( %h_{j}\right) \in \mathbb Z^{N},l=\left( %l_{j}\right) \in 
%\mathbb Z
%^{N},$ one has 
%$$\underset{\varepsilon \rightarrow 0}{\lim %}\int_{\mathbb R_{x}^{N}}\gamma %_{h,l}^{\varepsilon }\left( x\right) \varphi %\left( x\right)
%dx=0,$$ for all $\varphi \in L^{1}\left(
%\mathbb R _{x}^{N}\right) $ that is $\gamma %_{h,l}^{\varepsilon }\longrightarrow
%M\left( \gamma _{h,l}\right) =0$ in $L^{\infty %}\left(
%\mathbb R_{x}^{N}\right) -weak\ast $ when %$\varepsilon \rightarrow 0.$
%{\bf where do we use this part above?}

\medskip
Following \cite{nnang reit}, for the given ${\rm N}-$function $B$, by $\Xi ^{B}\left( \mathbb{R}
_{y}^d;\mathcal C_b\left( 
\mathbb{R}
_{z}^d\right) \right) $ %(or simply $\Xi ^{B}\left( 
%\mathbb{R}
%_{y}^d;\mathcal C_b\right) $) 
we denote the space \begin{align}\label{Xi}
\Xi ^{B}\left( 
\mathbb{R}_{y}^d;\mathcal C_b(\mathbb R^d_z)\right):=
\Big\{ u \in L_{loc}^{B}\left( 
\mathbb{R}
_{x}^d;\mathcal C_b\left( 
\mathbb{R}
_{z}^d\right) \right): \hbox{for every } U \in {\mathcal A}(\mathbb R^d_x): \nonumber\\
\underset{0<\varepsilon \leq 1}{\sup }\inf \Big\{
k>0,\int_{U}B\left( \tfrac{\left\Vert u\left( \frac{x}{\varepsilon },\cdot\right)
	\right\Vert _{L^{\infty }}}{k}\right) dx\leq 1\Big\} <\infty \Big\},
\end{align}
which, endowed with the norm
\begin{align}\label{ornLBPer}
&\left\Vert u\right\Vert _{\Xi ^{B}\left( 
	\mathbb{R}
	_{y}^d;\mathcal C_b\left( 
	\mathbb{R}
	_{z}^d\right) \right) }:= \\
\nonumber &\underset{0<\varepsilon \leq 1}{\sup}\inf \Big\{
k>0,\int_{B_d(\omega,1)}B\Big( \tfrac{\left\Vert u\left( \frac{x}{\varepsilon },\cdot\right) \right\Vert _{L^{\infty }}}{k}\Big) dx\leq 1\Big\} ,
\end{align}%
turns out to be a Banach space ($B_d(\omega,1)$, above being the unit ball of $
\mathbb{R}_x^d$ centered at the origin).
 %we have a norm on \ $\Xi ^{B}\left( 
%\mathbb{R}
%_y^N;\mathcal C_b\left(
%\mathbb{R}_z^N\right) \right) $ which %makes it a Banach space. 

\noindent We denote by $\mathfrak{X}_{per}
^{B}\left( 
\mathbb R _{y}^d;\mathcal C_b(\mathbb R^d_z)\right) $ the closure of $\mathcal{C}_{per}\left(
Y\times Z\right) $\ in \ $\Xi ^{B}\left(
\mathbb R
_{y}^d;\mathcal C_b(\mathbb R^d_z)\right)$. and with the above notation, by $L_{per}^{B}\left( Y\times Z\right) $ the space of functions in $L^B_{\rm loc}(\mathbb R_y^d\times \mathbb R^d_z)$ which are $Y\times Z$-periodic, with norm
 $\left\Vert \cdot\right\Vert _{B,Y\times Z}$, 
 %is a norm on $L_{per}^{B}\left(
%Y\times Z\right) $, namely it suffices to 
(i.e. one considers the $L^B$ norm just on  the unit period).
Furthermore it is immediately seen that 
\begin{equation*}
\left\vert \int_{B_d(\omega,1)}u\left( \frac{x}{\varepsilon },\frac{x}{\varepsilon
	^{2}}\right) dx\right\vert \leq \int_{B_d(\omega,1)}\left\Vert u\left( \frac{x}{%
	\varepsilon },\cdot\right) \right\Vert _{\infty }dx\leq 2\left\Vert 1\right\Vert
_{\widetilde{B},B_d(\omega,1)}\left\Vert u\right\Vert _{\Xi^{B}\left(
	\mathbb R_y^d;\mathcal C_b\left( 
	\mathbb R_z^d\right) \right) },
\end{equation*}
for every $u\in \mathcal{C}_{per}\left( Y\times Z\right) $ and for every $0<\varepsilon \leq 1$.

The following results, useful to prove estimates which involve test functions on oscillating arguments (see for instance Proposition \ref{propcomp}), can be found in \cite[Section 2]{FNZOpuscula}.
%  is a preliminary instrument which aims at comparing the $L^B$ norm in $Y\times Z$ with the one in \eqref{originalnormLBPer}.

\begin{lemma}\label{lemma2.2}
	There exists $C\in \mathbb R^+$ such that $\left\Vert u^{\varepsilon }\right\Vert
	_{B,B_d(\omega,1)}\leq C \left\Vert u\right\Vert
	_{B,Y\times Z },$ for every $0<\varepsilon \leq 1$, and $u\in \mathfrak{X}_{per}^{B}\left( 
	\mathbb R
	_{y}^d;\mathcal C_b(\mathbb R^d_z)\right).$ 
\end{lemma}

\begin{lemma}
	The operator $M$ defined on $\mathcal{C}_{per}\left( Y\times Z\right) $ by \eqref{M}
	%$$
	%u \in \mathcal C_{per}(Y\times Z)\to %M(u)=\iint_{Y\times Z}u(y,z)dydz \in \mathbb R,
	%$$ 
	can be
	extended (with the same notation) by continuity to a unique linear and continuous operator
	%denoted in
	%the same way 
	from $\mathfrak{X}_{per}^{B}\left( 
	\mathbb R_y^d;\mathcal C_b(\mathbb R^d_z)\right) $\ to $
	\mathbb R$ in such a way that it results non negative and  translation invariant.
\end{lemma}

 Finally we recall that $\mathfrak{X}_{per}^{B}\left(\mathbb R_y^d;\mathcal C_b(\mathbb R^d_z)\right) $ can be endowed with another norm, considering the set $\mathfrak{X}%
_{per}^{B}\left( 
\mathbb{R}_{y}^d\times 
\mathbb{R}_{z}^d\right) $ as the closure of $\mathcal{C}_{per}\left( Y\times Z\right) $
in $L_{loc}^{B}\left( 
\mathbb{R}_{y}^d\times 
\mathbb{R}
_{z}^d\right) $ with the norm 
\begin{equation*}\left\Vert u\right\Vert _{\Xi^{B}}:= \sup_{0<\varepsilon \leq 1}\left\|u\left(\frac{x}{\varepsilon}, \frac{y}{\varepsilon}\right)\right\|_{B, (B_d(\omega,1))^2}.
\end{equation*}

Via Riemann-Lebesgue lemma it can be proven that the above norm is equivalent to 
$\|u\|_{L^B(Y\times Z)},
$ thus in the sequel we will consider this one. 
For completeness, we state the following result, whose proof is in the Appendix of \cite{FNZOpuscula}. It states that the latter norm is controlled by the one defined in \eqref{ornLBPer}, thus together with Lemma \ref{lemma2.2}, it provides the eqivalence among the introduced norms in $\mathfrak{X}_{per}^B(\mathbb R^d_y;\mathcal C_b(\mathbb R^d_z))$.
%The proof is postponed \color{blue} to \color{black} the Appendix.

\begin{proposition}\label{prop2.1}
	It results that
	$\mathfrak{X}_{per}^{B}\left(
	\mathbb R_y^d;\mathcal C_b(\mathbb R^d_z)\right) \subset L_{per}^{B}\left( Y\times Z\right) =
	\mathfrak{X}_{per}^{B}\left( 
	\mathbb R_y^d\times 
	\mathbb R_z^d\right) $ and $\left\Vert u\right\Vert _{B,Y \times Z }\leq c\left\Vert
	u\right\Vert _{\Xi ^{B}\left(
		\mathbb R_y^d;\mathcal C_b\left( 
		\mathbb R_z^d\right) \right) }$ for all $u\in \mathfrak{X}_{per}^{B}\left( 
	\mathbb R_y^d;\mathcal C_b(\mathbb R^d_z)\right) .$
\end{proposition}

\subsection{Reiterated two-scale convergence in Orlicz spaces}
\label{reitdef}
In this subsection we recall the results proven in \cite{FNZOpuscula}, which, in turn, generalize  on the one hand to the Orlicz setting the notions introduced in \cite{ngu0} and, on the other, to the multiscale setting the results introduced in \cite{fotso nnang 2012, nnang reit, Nnang These} (cf. \cite{All1, All2, FZ, V} among a wide literature in the Sobolev setting). 
For the sake of exposition and having in mind the applications to remainder on the paper, we assume, within the section, that $B$ and $\tilde B$ satisfy \eqref{Delta2}.

%Via generalisation of definitions in $\left[ 16,19,33\right] \bigskip $
%define%

\noindent Recalling the spaces introduced in subsection \ref{O-Sspaces}, we start by defining reiterated two-scale convergence:

\begin{definition}\label{def3s}
	A sequence of functions $\left( u_{\varepsilon }\right) _{\varepsilon}  \subseteq L^{B}\left( \Omega \right) $ is said to be:
	\begin{itemize}
		\item[-]weakly reiteratively two-scale convergent in $L^{B}\left( \Omega \right) $
		to $u_{0}\in L^{B}\left(\Omega;L^B_{per}(Y\times
		Z)\right) $ if 
		\begin{equation}
		\label{2}
		\int_{\Omega }u_{\varepsilon }f^{\varepsilon }dx\rightarrow 
		\iiint_{\Omega \times Y\times Z}u_{0}fdxdydz, \hbox{ for all } f\in L^{%
			\widetilde{B}}\left( \Omega ;\mathcal{C}_{per}\left( Y\times Z\right) \right),
		\end{equation}%
		as $\varepsilon \to 0$,
		\item[-]strongly reiteratively two-scale convergent in $L^{B}\left( \Omega \right) $\
		to $u_{0}\in L^{B}\left( \Omega;L^B_{per}( Y\times Z)\right) $\ if for $%
		\eta >0$ and $f\in L^{B}\left( \Omega ;\mathcal{C}_{per}\left( Y\times
		Z\right) \right) $ verifying $\left\Vert u_{0}-f\right\Vert _{B,\Omega
			\times Y\times Z}\leq \frac{\eta }{2}$ there exists $\rho >0$ such that $%
		\left\Vert u_{\varepsilon }-f^{\varepsilon }\right\Vert _{B,\Omega }\leq
		\eta $ for all $0<\varepsilon \leq \rho.$
	\end{itemize}
\end{definition}

When (\ref{2}) happens %for all $f\in L^{\widetilde{B}}\left( \Omega ;%
%\mathcal{C}_{per}\left( Y\times Z\right) \right) $ 
we denote it by "$%
u_{\varepsilon }\rightharpoonup u_{0}$ in $L^{B}\left( \Omega
\right)-$\ weakly reiteratively two-scale " 
%or simply " $u_{\varepsilon }\rightarrow $ $u_{0}$
%weakly reiteratilvely in $L^{B}\left( \Omega \right) $\ two-scale" 
and we say that 
$u_{0}$ is the weak reiterated two-scale limit in $L^{B}\left( \Omega
\right) $ of the sequence $\left( u_{\varepsilon }\right) _{\varepsilon}.$

\begin{remark}
	The above definition is given in the scalar setting, but it extends to vector valued functions, arguing in components.
\end{remark}

The following results have been proven in \cite[Subsection 2.3]{FNZOpuscula}. We report them since they will be used in the sequel.

\begin{proposition}
	If $u\in L^{B}\left( \Omega ;\mathcal{C}_{per}\left( Y\times Z\right)
	\right) $ then (with the notation in subsection \ref{traces}) $u^{\varepsilon }\rightharpoonup  $u in $L^{B}\left(
	\Omega \right) $ weakly reiteratively two-scale, and we have $\underset{\varepsilon
		\rightarrow 0}{\lim }\left\Vert u^{\varepsilon }\right\Vert _{B,\Omega
	}=\left\Vert u\right\Vert _{B,\Omega \times Y\times Z}.$
\end{proposition}

%\begin{proof}
%	Let $u\in L^{B}\left( \Omega ;\mathcal{C}_{per}\left( Y\times Z\right)
%	\right) $ and $f\in L^{\widetilde{B}}\left( \Omega ;\mathcal{C}_{per}\left(
%	Y\times Z\right) \right) $ then $uf\in L^{1}\left( \Omega ;\mathcal{C}_{per}\left( Y\times Z\right) \right) $ and  \begin{equation*}\lim_{\varepsilon
%		\rightarrow 0} \int_{\Omega }u^{\varepsilon }f^{\varepsilon
%	}dx=\iiint_{\Omega \times Y\times Z}ufdxdydz.
%	\end{equation*} Similary $%
%	\hbox{ for all } \delta >0,B\left( \left|\frac{u}{\delta }\right|\right) \in L^{1}\left( \Omega ;%
%	\mathcal{C}_{per}\left( Y\times Z\right) \right) $ and the result follows.
%\end{proof}
%
Next we recall a sequential compactness result in the Orlicz setting.

\begin{proposition}\label{propcomp}
	Given a bounded sequence $\left(
	u_{\varepsilon}\right)_{\varepsilon}\subset L^{B}\left( \Omega
	\right),$ one can extract a not relabeled subsequence such that $\left(u_{\varepsilon }\right)_{\varepsilon}$ is
	weakly reiteratively two-scale convergent in $L^{B}\left( \Omega \right) .$
\end{proposition}

The results in the sequel follow (similarly to the (non reiterated) case in \cite{fotso nnang 2012})) as a consequence of density results in the `standard' setting.

\begin{proposition}
	If a sequence $\left( u_{\varepsilon }\right) _\varepsilon$ is weakly
	reiteratively two-scale convergent in $L^{B}\left( \Omega \right) $
	to $u_0\in L{B}\left( \Omega;L^B_{per}( Y\times Z)\right) $ then
	\begin{itemize}
		\item[(i)]  $u_{\varepsilon }\rightharpoonup\int_{Z}u_0\left(
		\cdot,\cdot,z\right) dz$ in $L^{B}\left( \Omega \right) $ weakly two-scale,
		and
		\item[(ii)]  $u_{\varepsilon }\rightharpoonup \widetilde{u_0}$ in $L^{B}\left(
		\Omega \right) $-weakly as $\varepsilon \to 0$ where $\widetilde{u_0}
		\left( x\right) =\iint_{Y\times Z}u_0\left( x,\cdot,\cdot\right) dydz.$
	\end{itemize}
\end{proposition}

In the sequel we will consider the space 
\begin{align}\label{Xiinfty}\mathfrak{X}_{per}^{B
		,\infty }\left( \mathbb{R}_{y}^{d},\mathcal{C}_b\left( 
	\mathbb{R}
	_{z}^{d}\right) \right)
	:=\mathfrak{X}_{per}^{B
	}\left( 
	\mathbb{R}
	_{y}^{d},\mathcal{C}_b\left( 
	\mathbb{R}
	_{z}^{d}\right) \right) \cap L^{\infty }\left( 
	\mathbb{R}
	_{y}^{d},\mathcal{C}_b\left( 
	\mathbb{R}
	_{z}^{d}\right) \right),
\end{align} endowed with the $L^\infty$ norm.

\begin{proposition}\label{prop2.4}
%	Let	$\mathfrak{X}_{per}^{B,\infty}\left( 
%	\mathbb R_y^d;\mathcal C_b\right):=\mathfrak{X}_{per}^B\left( 
%	\mathbb R_y^d;\mathcal C_b\right)\cap L^\infty(\mathbb R^d_y\times \mathbb R^d_z).$
	If $\left( u_{\varepsilon }\right) _{\varepsilon}$ is
	weakly reiteratively two-scale convergent in $L^{B}\left( \Omega \right) $
	to $u_{0}\in L^{B}\left( \Omega;L^B_{per} (Y\times Z)\right) $ then
	$\int_{\Omega }u_{\varepsilon }f^{\varepsilon }dx\rightarrow
	\iiint_{\Omega \times Y\times Z}u_{0}fdxdydz,$ for all $f\in \mathcal{C}\left( 
	\overline{\Omega }\right) \otimes \mathfrak{X}_{per}^{B,\infty}\left( 
	\mathbb R_y^d;\mathcal C_b(\mathbb R^d_z)\right) .$

%\begin{corollary}\label{corollary2.1}
Moreover, if $v\in \mathcal{C}\left( \overline{\Omega };\mathfrak{X}_{per}^{B,\infty
	}(\mathbb R_y^d;\mathcal C_b(\mathbb R^d_z)) \right),$ then $v^{\varepsilon } \rightharpoonup v$ in $L^{B}\left( \Omega \right) $- weakly reiteratively two-scale as $\varepsilon
	\rightarrow 0.$
\end{proposition}

\begin{remark}
	\begin{itemize}
		\item[(1)] If $v\in L^{B}\left( \Omega ;\mathcal{C}_{per}\left( Y\times Z\right)
		\right) ,$ then $v^{\varepsilon }\rightarrow v$ in $L^{B}\left( \Omega
		\right) $-strongly reiteratively two-scale as $\varepsilon \rightarrow 0.$
		
		\item[(2)] If $\left( u_{\varepsilon }\right) _{\varepsilon}\subset
		L^{B}\left( \Omega \right) $ is strongly reiteratively two-scale convergent in $
		L^{B}\left( \Omega \right) $\ to $u_{0}\in L^{B}\left( \Omega;L^B_{per}(Y\times Z)\right)$ then 
		\begin{itemize}
			\item[(i)]  $u_{\varepsilon }\rightharpoonup u_{0}$ in $L^{B}\left(
			\Omega \right) $ weakly reiteratively two-scale as $\varepsilon \rightarrow 0;$
			\item[(ii)] $\left\Vert u_{\varepsilon }\right\Vert _{B,\Omega }\rightarrow
			\left\Vert u_{0}\right\Vert _{B,\Omega \times Y\times Z}$ as $
			\varepsilon \rightarrow 0.$
		\end{itemize}
	\end{itemize}
\end{remark}

The following result is key to define weak reiterated two-scale convergence in Orlicz-Sobolev spaces, also providing a sequential compactness result in $
W^{1}L^{B}\left( \Omega \right).$ The proof, which extends with alternative arguments \cite[Theorem 4.1]{fotso nnang 2012}, can be found in \cite[Proposition 2.12]{FNZOpuscula}, see also \cite[Remark 2]{FTGNZ} for the more regularity stated below.

Indeed, recalling that  $L_{per}^B\left(
\Omega;W_{\#}^{1}L^{B}\left( Y\right) \right)$ denotes the space of functions $u \in L^B(\Omega;L^B_{per}( Y))$, such that $u(x,\cdot) \in W_{\#}^{1}L^{B}\left( Y\right) $, for a.e. $x \in \Omega$ and $L_{per}^B\left(
Y;W_{\#}^{1}L^{B}\left( Z\right) \right)$ denotes the space of functions $u \in L^B_{per}(Y\times Z)$, such that $u(y,\cdot) \in W_{\#}^{1}L^{B}\left( Z\right) $, for a.e. $y \in Y$, we have
\color{black}
\begin{proposition}\label{mainprop3s}
	Let $\Omega $ be a bounded open set in $
	\mathbb R_x^d$, and $\left( u_{\varepsilon
	}\right) _{\varepsilon}$ bounded in $W^{1}L^{B}\left( \Omega \right).$ 
	There exist a not relabeled subsequence, $u_{0}\in W^{1}L^{B}\left(
	\Omega \right),$ $\left(u_{1},u_{2}\right) \in L^B_{per}\left(
	\Omega ;W_{\#}^{1}L^{B}\left( Y\right) \right) \times L^B\left(
	\Omega ;L^B_{per}\left( Y;W_{\#}^{1}L^{B}\left( Z\right) \right)
	\right) $ such that:\color{black}
	\begin{itemize}
		\item[(i)] $u_{\varepsilon }\rightharpoonup u_{0}$ weakly reiteratively two-scale in $L^{B}\left(
		\Omega \right) $,
		\item[(ii)] $D_{x_{i}}u_{\varepsilon }\rightharpoonup
		D_{x_{i}}u_{0}+D_{y_{i}}u_{1}+D_{z_{i}}u_{2}$ weakly reiteratively two-scale in $L^{B}\left( \Omega\\
		\right) $, $1\leq i\leq d$,
	\end{itemize}
	as $\varepsilon \to 0$.
\end{proposition}

\begin{corollary}
	If $\left( u_{\varepsilon }\right) _{\varepsilon}$ is such that$\ \
	u_{\varepsilon }\rightharpoonup v_{0}$ weakly reiteratively two-scale in $W^{1}L^{B}\left( \Omega \right)$,  we have:
	\begin{itemize}
		\item[(i)] 	$u_{\varepsilon }\rightharpoonup
		\int_{Z}v_0\left( \cdot ,\cdot ,z\right) dz$ weakly two-scale in $W^{1}L^{B}\left( \Omega \right) $, 
		\item[(ii)] $u_{\varepsilon }\rightharpoonup 
		\widetilde{v_0}$ in $W^{1}L^{B}\left( \Omega \right) $-weakly, where $\widetilde{v_0}\left( x\right) =
		\iint_{Y\times Z}v_0\left( x,\cdot ,\cdot\right) dydz.$ 
	\end{itemize}  
\end{corollary}

In view of the next applications, we underline that, under the assumptions of the above proposition, the canonical injection $%
W^{1}L^{B}\left( \Omega \right) \hookrightarrow L^{B}\left( \Omega \right) $
is compact.

\begin{remark}\label{rem 2.13}
If $\left( v_{\varepsilon }\right) _{\varepsilon }\subset L^{B}\left(
\Omega \right) $ and $v_{\varepsilon }\rightharpoonup  v_{0}$ weakly reiteratively two-scale in $%
L^{B}\left( \Omega \right) $, as $\varepsilon \rightarrow 0$%
\; and for any $\varepsilon >0: v_{\varepsilon }\geq 0$ $a.e$ in $%
\Omega \times Y\times Z$ then $v_{0}\geq 0$ $a.e$ in $\Omega \times Y\times
Z.$
\end{remark}

\color{black}

\section{Weak solution of \eqref{1.1}}\label{solweak}

\noindent The aim of this section consists of showing the existence of a weak solution to \eqref{1.1}, neglecting the periodicity assumptions on $a$.
 
Given $\left( v,\mathbf{V}\right) =\left( v,\left( v_{i}\right) \right) \in 
\mathcal{C}\left( \overline{\Omega }
\right) \times \mathcal{C}\left( \overline{\Omega }\right)^{d}$ and $1\leq i\leq d,$ one can check, using assumptions $\left( H_{1}\right)
-\left( H_{4}\right),$ that the function $\left( x,y,z\right) \rightarrow
a_{i}\left( y,z,v\left( x\right) ,\mathbf{V}\left( x\right) \right) $ of $%
\overline{\Omega }\times 
\mathbb{R}
^{d}\times 
\mathbb{R}
^{d}$ into $%
\mathbb{R}
$ belongs to $\mathcal{C}\left( \overline{\Omega };L^{\infty }\left( 
\mathbb{R}
_{y}^{d};\mathcal{C}_{per}\left( Z\right) \right) \right) .$ Hence for fixed 
$\varepsilon >0$, the function $x\rightarrow a_{i}\left( \frac{x}{\varepsilon 
},\frac{x}{\varepsilon ^{2}},v\left( x\right) ,\mathbf{V}\left( x\right)
\right) $ of $\Omega $ into $%
%TCIMACRO{\U{211d} }%
%BeginExpansion
\mathbb{R}
%EndExpansion
$ denoted by $a_{i}^{\varepsilon }\left( \cdot, \cdot,v,\mathbf{V}\right) $ is well
defined by a function in $L^{\infty }\left( \Omega 
\right) $ see \cite[Remark 2.1]{nnang reit}. In particular we have the following:

\begin{proposition}\label{propalphabeta}
	Let $(a_i)_{1\leq i \leq d}$ be the functions in \eqref{1.1} and assume that they satisfy the assumptions $(H_1)-(H_4)$. Let $a^{\varepsilon }\left( \cdot, \cdot,v,%
	\mathbf{V}\right) =\left( a_{i}^{\varepsilon }\left( \cdot, \cdot,v,\mathbf{V}\right)
	\right) _{1\leq i\leq d},$ then
the transformation $\left( v,\mathbf{V}\right) \rightarrow a^{\varepsilon
}\left( \cdot, \cdot,v,\mathbf{V}\right) $ of $\ \mathcal{C}\left( \overline{\Omega }
\right) \times \mathcal{C}\left( \overline{\Omega }
\right) ^{d}$ into $L^{\infty }(\Omega)^{d}$ extends by continuity to a mapping, still denoted by $\left( v,%
\mathbf{V}\right) \rightarrow a^{\varepsilon }\left( \cdot, \cdot,v,\mathbf{V}\right)
,$ from $L^{\Phi }\left( \Omega
\right) \times L^{\Phi }\left( \Omega
\right) ^{d}$ into $L^{\widetilde{\Phi }}\left( \Omega
\right) ^{d}$ such that
\begin{equation}\label{Lest}
\left\Vert a^{\varepsilon }\left( \cdot, \cdot,v,\mathbf{V}\right) -a^{\varepsilon
}\left( \cdot, \cdot,w,\mathbf{W}\right) \right\Vert _{\widetilde{\Phi },\Omega }\leq
c\left\Vert v-w\right\Vert _{\Phi ,\Omega }^{\alpha }+c^{\prime }\left\Vert
\bf V-\bf W\right\Vert _{\Phi ,\Omega }^{\beta },  
\end{equation}
for all $\left( v,\mathbf{V}\right) ,\left( w,\mathbf{W}\right) \in L^{\Phi
}\left( \Omega
\right) \times L^{\Phi }\left( \Omega
\right) ^{d}$ where $c,c^{\prime }>0$ and $\alpha ,\beta \in \left\{ \frac{%
\rho _{1}}{\rho _{0}}\left( \rho _{0}-1\right) ,\frac{\rho _{2}}{\rho _{0}}%
\left( \rho _{0}-1\right) ,\right.$ 

\noindent$\left.\rho _{1}-1,\frac{\rho _{2}}{\rho _{1}}\left(
\rho _{1}-1\right) \right\},$
with the constants $\rho_0,\rho_1$ and $\rho_2$ as in \eqref{1.2}.
\end{proposition}

\begin{proof}
The proof is entirely similar to the one of \cite[Proposition 2.3]{Nnang Orlicz 2014}, relying in turn on \cite{Cle, KP}, which we refer to. Indeed, also the constants' values  $\alpha$ and $\beta$ are deduced according to the value of the Luxemburg norm on the right hand side of \eqref{Lest}.
\end{proof}

\color{black}

By $(i)$ and $(ii)$ in $(H_1)$, the assumptions on $f$, \eqref{1.3} in $(H_2)$, $(H_3)$ and $(H_4)$ it follows that for every $\varepsilon >0$, the assumptions of \cite[Theorem 3.2]{Y} are satisfied, hence there exists $%
u_{\varepsilon }\in W_{0}^{1}L^{\Phi }\left( \Omega
\right) \cap L^{\infty }\left( \Omega\right) $ weak solution of \eqref{1.1}, i.e. \begin{align}
	\exists  u_{\varepsilon }\in W_{0}^{1}L^{\Phi }\left( \Omega
	\right) \cap L^{\infty }\left( \Omega\right) \hbox{ such that } \label{eqvarepsilon}\\
	\int_{\Omega }a\left( \frac{x}{\varepsilon },\frac{x}{%
\varepsilon ^{2}},u_{\varepsilon },Du_{\varepsilon }\right)
Dvdx=\int_{\Omega }fvdx, \hbox{ for every }v\in  W_{0}^{1}L^{\Phi }\left( \Omega
\right). \nonumber
\end{align} Thus, by \eqref{1.4} in $(H_4)$ and $(H_5)$, it follows that 
\begin{align*}\int_{\Omega }\theta \Phi \left( \left\vert
Du_{\varepsilon }\right\vert \right) dx\leq \int_{\Omega }a\left( \frac{x}{%
\varepsilon },\frac{x}{\varepsilon ^{2}},u_{\varepsilon },Du_{\varepsilon
}\right) Du_{\varepsilon }dx=\int_{\Omega }fu_{\varepsilon }dx\leq
2\left\Vert f\right\Vert _{\widetilde{\Phi },\Omega }\left\Vert
u_{\varepsilon }\right\Vert _{\Phi ,\Omega },
\end{align*}
where 
\begin{align}
	\label{theta}
\theta:= {\widetilde \Phi^{-1}}(\Phi (\min_{t\geq 0} h(t)) )> 0.
\end{align}
If $\left\Vert \left\vert Du_{\varepsilon }\right\vert \right\Vert _{\Phi
,\Omega }\geq 1,$we have 
\begin{align*}\theta \left\Vert \left\vert Du_{\varepsilon
}\right\vert \right\Vert _{\Phi ,\Omega }^{\sigma }\leq \int_{\Omega }\theta
\Phi \left( \left\vert Du_{\varepsilon }\right\vert \right) dx\leq
2\left\Vert f\right\Vert _{\widetilde{\Phi },\Omega }\left\Vert
u_{\varepsilon }\right\Vert _{\Phi ,\Omega }\leq c\left\Vert \left\vert
Du_{\varepsilon }\right\vert \right\Vert _{\Phi ,\Omega },\sigma >1.
\end{align*} 
We deduce therefore that: $$\underset{0<\varepsilon }{\sup }\left\Vert
 u_{\varepsilon } \right\Vert _{W_{0}^{1}L^{\Phi
}\left( \Omega
\right) }<+\infty .$$

It is easily observed that if $(H_6)$ is satisfied the solution in \eqref{eqvarepsilon} is unique. Indeed, suitable differences between the weak formulations of \eqref{1.1} with two possible solutions allow us to get uniqueness. We refer to Remark \ref{remuniq} for more details regarding uniqueness of the limiting problem.

\section{Homogenization problem}\label{hompbsec}

We intend to investigate the limiting behaviour, as $0<\varepsilon
\rightarrow 0,$ of a sequence of solutions $u_{\varepsilon }$ of \eqref{1.1}, i.e. $u_\varepsilon$ as in \eqref{eqvarepsilon}, taking into account all the assumptions in section \ref{sec1}, in particular requiring that the coefficients $a$ satisfy $(H_1)-(H_4)$ together with the periodicity assumption  $(H_5).$  Hence, we start by showing in suitable spaces of regular functions, where the compositions of functions in the weak formulation of \eqref{1.1} are meaningful, which convergence is appropriate to consider limits as $\varepsilon \to 0$.

For the readers' convenience we put in the Appendix some intermediate results, dealing with the well posedness of the compositions appearing in \eqref{1.1}.  

Indeed, taking into account the just mentioned traces results, making use of the notation in Section \ref{notations} (cf. \eqref{Xiinfty}), under the extra assumption of periodicity on the first two variables of $a$ and for $(w, {\mathbf W})$ periodic, we conclude that, for every $1\leq i \leq d$, the map
$x\rightarrow a_{i}\left(
\cdot, \cdot,w\left( x,\cdot,\cdot\right) ,\mathbf{W}\left( x,\cdot,\cdot\right) \right) $ belongs to $\mathcal{C}\left( \overline{\Omega };\mathfrak{X}_{per}^{\widetilde{\Phi }%
,\infty }\left( \mathbb{R}_{y}^{d},\mathcal{C}_b\left( 
\mathbb{R}
_{z}^{d}\right) \right) \right)$,
where $a_{i}\left( \cdot, \cdot,f \left(
\cdot,\cdot\right) ,\mathbf{f }\left( \cdot,\cdot\right) \right) $ is the function $%
\left( y,z\right) \rightarrow a_{i}\left( y,z,f \left( y,z\right),
\mathbf{f}\left( y,z\right) \right),$ with $\left( f ,\mathbf{f}\right)
\in $
$\mathcal{C}_{per}\left( Y\times Z\right) ^{d+1}.$

Hence, we are in position to state our first convergence result, prior to the proof of our main existence theorem.

\color{black}
\begin{proposition}\label{prop3.3}
For $\left( f,\mathbf{f }\right) \in %{\color{magenta} \mathcal{C}_{per}\left( Y\times %Z\right) ^{d+1}} $ %{\color{magenta} I would put the %previous space instead of 
\mathcal{C}\left( \overline{%
\Omega };\mathcal{C}_{per}\left( Y\times Z\right) ^{d+1}\right)$, let $a= (a_i)_{1\leq i \leq d}: \mathbb R^d \times \mathbb R^d \times \mathbb R \times \mathbb R^d \to \mathbb R^d$ satisfy  $(H_1)-(H_5)$, we have that 
\begin{itemize}
	\item[(i)] for each $1\leq i\leq
	d,$
\begin{align*} a_{i}\left( y, z,f \left(x, y,z\right) ,\mathbf{f}%
\left(x, y,z\right) \right) \in \mathcal{C}\left( \overline{\Omega };\mathfrak{%
X}_{per}^{\widetilde{\Phi },\infty }\left( 
\mathbb{R}
_{y}^{d},\mathcal{C}_b\left( 
\mathbb{R}
_{z}^{d}\right) \right) \right) 
\;\hbox{ and}
\\
a_{i}^{\varepsilon }\left( \cdot, \cdot,f ^{\varepsilon }(\cdot, \cdot,\cdot),\mathbf{f }%
^{\varepsilon }(\cdot, \cdot,\cdot)\right) \rightharpoonup a_{i}\left( \cdot, \cdot,f \left(\cdot,\cdot,\cdot\right) ,%
\mathbf{f }\left(\cdot, \cdot,\cdot\right) \right) %
\text{weakly reiteratively two-scale}, 
\nonumber\\
\text{ in }L^{\widetilde{\Phi }}(\Omega), \text{ as }\varepsilon \rightarrow 0.  \nonumber 
\end{align*}

\item[(ii)]
The function $\left( f(\cdot, \cdot,\cdot) ,\mathbf{f }(\cdot, \cdot,\cdot)\right)
\rightarrow a\left( \cdot, \cdot,f ,\mathbf{f }\right) $ from $\mathcal{C}%
\left( \overline{\Omega };\mathcal{C}_{per}\left( Y\times Z\right)
^{d+1}\right) $\ to 

\noindent $L^{\widetilde{\Phi }}\left( \Omega ;L_{per}^{\widetilde{%
\Phi }}\left( Y\times Z\right) \right) ^{d}$\ \ extends by continuity to a (not relabeled)
function from $L^{\Phi }\left( \Omega; L_{per}^{\Phi }\left(
Y\times Z\right) \right) ^{d+1}$\ \ to \ $L^{\widetilde{\Phi }}\left( \Omega
;L_{per}^{\widetilde{\Phi }}\left( Y\times Z\right) \right) ^{d}$, and there exist positive constants $c, c'$ such that: 
\begin{align*}
\left\Vert a\left( \cdot, \cdot,u((\cdot, \cdot,\cdot)),\mathbf{U}(\cdot, \cdot,\cdot)\right) -a\left( \cdot, \cdot,v(\cdot, \cdot,\cdot),\mathbf{V}(\cdot, \cdot,\cdot)\right)
\right\Vert _{L^{\widetilde{\Phi }}\left( \Omega ;L_{per}^{\widetilde{\Phi }%
}\left( Y\times Z\right) \right) ^{d}}\leq  
\nonumber\\ 
c\left\Vert u(\cdot, \cdot,\cdot)-v(\cdot, \cdot,\cdot)\right\Vert _{L^{\Phi }\left( \Omega ;L_{per}^{\Phi }\left(
Y\times Z\right) \right) }^{\alpha }+c^{\prime }\left\Vert \mathbf{U(\cdot, \cdot,\cdot)-V(\cdot, \cdot,\cdot)}%
\right\Vert _{L^{\Phi }\left( \Omega ;L_{per}^{\Phi }\left( Y\times Z\right)
^{d}\right) }^{\beta },
%\label{3.49}
\\
a\left( \cdot, \cdot,u(\cdot, \cdot,\cdot),\mathbf{U}(\cdot, \cdot,\cdot)\right) \cdot\mathbf{U}\geq \theta \Phi \left(
\left\vert \mathbf{U}(\cdot, \cdot,\cdot)\right\vert \right) \hbox{ a.e. in } \Omega \times \mathbb{R}^{d}\times\mathbb{R}
^{d} \label{3.5}
\end{align*}
\end{itemize}
and 
\begin{align*}
\left(a\left( \cdot, \cdot,u(\cdot, \cdot,\cdot)\mathbf{U}(\cdot, \cdot,\cdot)\right) -a\left( \cdot, \cdot,u(\cdot, \cdot,\cdot),\mathbf{V}(\cdot, \cdot,\cdot)\right) \right)\cdot\left( 
\mathbf{U}(\cdot, \cdot,\cdot)-\mathbf{V}(\cdot, \cdot,\cdot)\right) \geq 0
\\
\text{ a.e in }\Omega \times \mathbb{R}^{d}\times \mathbb{R}^{d}, 
\end{align*}%
for all $u,v\in L^{\Phi }\left( \Omega ;L_{per}^{\Phi }\left( Y\times
Z\right) \right) ,\mathbf{U},\mathbf{V}\in L^{\Phi }\left( \Omega
;L_{per}^{\Phi }\left( Y\times Z\right) ^{d}\right),$ where $\alpha ,\beta$ 
 and $\theta $ are defined as in Proposition \ref{propalphabeta} and \eqref{theta}, respectively, and $L^\Phi_{per}$ is as in subsection \ref{O-Sspaces}, with the $N$-function $B$ replaced by $\Phi$.
\end{proposition}

\begin{proof}
It is easily seen that, for every $1\leq i \leq d,$ when $\left( f,\mathbf{f}\right) \in \mathcal{C}\left( \overline{\Omega };%
\mathcal{C}_{per}\left( Y\times Z\right) ^{d+1}\right)$,
$a_{i}\left( \cdot, \cdot,f(\cdot, \cdot,\cdot),\mathbf{f}(\cdot, \cdot,\cdot)\right) $ lies in $\mathcal{C}\left( 
\overline{\Omega };\mathfrak{X}_{per}^{\widetilde{\Phi },\infty }(\mathbb R^d_z ;\mathcal C_b(\mathbb R^d_z))\right)$.
%,\left( \mathfrak{X}_{per}^{\widetilde{\Phi },\infty }\text{ equiped with
%the }L^{\infty }-norm\right),$ 

Thus, by Proposition \ref{prop2.4}, it follows that the sequence $a_{i}^{\varepsilon }\left(
\cdot, \cdot, f^{\varepsilon }(\cdot, \cdot,\cdot),\mathbf{f}^{\varepsilon }(\cdot, \cdot,\cdot)\right) $ $ \rightharpoonup
a_{i}\left( \cdot, \cdot,f \left(\cdot, \cdot,\cdot\right) ,\mathbf{f}\left( \cdot,\cdot,\cdot\right)
\right) $ weakly reiteratively two-scale  in $L^{\widetilde{\Phi }}(\Omega)$, as $\varepsilon \rightarrow
0.$

Moreover, by $(H_{1})-(H_{4})$ we get 
\begin{align*}
\left\vert a^{\varepsilon }\left( \cdot, \cdot,f^{\varepsilon }(\cdot, \cdot,\cdot),\mathbf{f}%
^{\varepsilon }(\cdot, \cdot,\cdot)\right) -a^{\varepsilon }\left( \cdot, \cdot,f^{\prime \varepsilon }(\cdot, \cdot,\cdot),%
\mathbf{f}^{\prime \varepsilon }(\cdot, \cdot,\cdot)\right) \right\vert \leq 
\\c_{1}\widetilde{%
\Psi }^{-1}\left( \Phi \left( c_{2}\left\vert f^{\varepsilon }(\cdot, \cdot,\cdot)-f^{^{\prime
}\varepsilon }(\cdot, \cdot,\cdot)\right\vert \right) \right) +
c_{3}\widetilde{\Phi }^{-1}\left( \Phi \left( c_{4}\left\vert \mathbf{f}%
^{\varepsilon }(\cdot, \cdot,\cdot)-\mathbf{f}^{\prime \varepsilon }(\cdot, \cdot,\cdot)\right\vert \right) \right)
; 
\\ 
a^{\varepsilon }\left( \cdot, \cdot,f^{\varepsilon }(\cdot, \cdot, \cdot),\mathbf{f}^{\varepsilon
}(\cdot, \cdot,\cdot)\right) \cdot \mathbf{f}^{\varepsilon }(\cdot, \cdot,\cdot)\geq \theta \Phi \left( \left\vert 
\mathbf{f}^{\varepsilon }(\cdot, \cdot,\cdot)\right\vert \right)  (\hbox{with }\theta \hbox{ as in }\eqref{theta});
\\
\left( a^{\varepsilon
}\left( \cdot, \cdot,f^{\varepsilon }(\cdot, \cdot,\cdot),\mathbf{f}^{\varepsilon }(\cdot, \cdot,\cdot)\right)
-a^{\varepsilon }\left( \cdot, \cdot,f^{ \varepsilon }(\cdot, \cdot,\cdot),\mathbf{f}^{\prime
\varepsilon }(\cdot, \cdot,\cdot)\right) \right) \cdot \left( \mathbf{f}^{\varepsilon }(\cdot, \cdot,\cdot)-\mathbf{f}%
^{\prime \varepsilon }(\cdot, \cdot,\cdot)\right) \geq 0,\\
 \hbox{ a.e. in } \Omega, \hbox{ for any }
\varepsilon >0.
\end{align*}
In particular,  arguing as in \cite[Proposition 3.4]{Nnang Orlicz 2014} (cf. also Remark \ref{rem 2.13}) it follows that for a.e $x$ in $\Omega$ and every $(y,z) \times 
\mathbb{R}^{d}\times 
\mathbb{R}
^{d},$
\begin{align*}
\left\vert a\left( \cdot, \cdot,f(\cdot,\cdot,\cdot),\mathbf{f}(\cdot,\cdot,\cdot)\right) -a\left( \cdot, \cdot,f^{\prime }(\cdot,\cdot,\cdot),\mathbf{%
f}^{\prime }(\cdot,\cdot,\cdot)\right) \right\vert \leq 
\\
c_{1}\widetilde{\Psi }^{-1}\left( \Phi
\left( c_{2}\left\vert f(\cdot,\cdot,\cdot)-f^{^{\prime }(\cdot,\cdot,\cdot)}\right\vert \right) \right) +
c_{3}\widetilde{\Phi }^{-1}\left( \Phi \left( c_{4}\left\vert \mathbf{f}(\cdot,\cdot,\cdot)-
\mathbf{f}^{\prime }(\cdot,\cdot,\cdot)\right\vert \right) \right) ;
\\
 a\left( \cdot, \cdot,f(\cdot,\cdot,\cdot),\mathbf{f}(\cdot,\cdot,\cdot)
\right) \cdot\mathbf{f}(\cdot,\cdot,\cdot)\geq \theta \Phi \left( \left\vert \mathbf{f}(\cdot,\cdot,\cdot)\right\vert
\right) ;
\\ 
\left( a\left( \cdot, \cdot,f(\cdot,\cdot,\cdot),\mathbf{f}(\cdot,\cdot,\cdot)\right) -a\left( \cdot, \cdot,f(\cdot, \cdot,\cdot),%
\mathbf{f}^{\prime }(\cdot,\cdot,\cdot)\right) \right) \cdot \left( \mathbf{f}(\cdot,\cdot,\cdot)-\mathbf{f}^{\prime
}(\cdot,\cdot,\cdot)\right) \geq 0, 
\end{align*} where $a\left( \cdot, \cdot,f(\cdot,\cdot,\cdot),\mathbf{f}(\cdot,\cdot,\cdot)\right) =\left( a_{i}\left( \cdot, \cdot,f(\cdot,\cdot,\cdot),%
\mathbf{f}(\cdot,\cdot,\cdot)\right) \right) _{1\leq i\leq d}.$
\color{black}
From the first of the above inequalities, 
\begin{align*}\int_{\Omega }\int_{Y}\int_{Z}%
\widetilde{\Phi }\left( \frac{\left\vert a\left( \cdot, \cdot,f(\cdot,\cdot,\cdot),\mathbf{f}(\cdot,\cdot,\cdot)\right)
-a\left( \cdot, \cdot,f^{\prime }(\cdot,\cdot,\cdot),\mathbf{f}^{\prime }(\cdot,\cdot,\cdot)\right) \right\vert }{\delta }%
\right) dxdydz\leq \\
\int_{\Omega }\int_{Y}\int_{Z}\widetilde{\Phi }\left( \frac{c_{1}\widetilde{%
\Psi }^{-1}\left( \Phi \left( c_{2}\left\vert( f-f^{\prime })(\cdot,\cdot,\cdot)\right\vert
\right) \right) +c_{3}\widetilde{\Phi }^{-1}\left( \Phi \left(
c_{4}\left\vert (\mathbf{f}-\mathbf{f}^{\prime })(\cdot,\cdot,\cdot)\right\vert \right) \right) }{%
\delta }\right) dxdydz,
\end{align*} for $\delta >0.$ Therefore, arguing  analogously to \cite{Nnang Orlicz 2014} and to Proposition \ref{propalphabeta}, we are lead to
\begin{align*}
\left\Vert a\left( \cdot, \cdot,f(\cdot,\cdot,\cdot),\mathbf{f}(\cdot,\cdot,\cdot)\right) -a\left( \cdot, \cdot,f^{\prime }(\cdot,\cdot,\cdot),\mathbf{%
f}^{\prime }(\cdot,\cdot,\cdot)\right) \right\Vert _{L^{\widetilde{\Phi }}\left( \Omega
;L_{per}^{\widetilde{\Phi }}\left( Y\times Z\right) \right) ^{d}}\leq 
\\
c\left\Vert f-f^{\prime }\right\Vert _{L^{\Phi }\left( \Omega
;L_{per}^{\Phi }\left( Y\times Z\right) \right) }^{\alpha }+c^{\prime
}\left\Vert \mathbf{f}-\mathbf{f}^{\prime }\right\Vert _{L^{\Phi }\left(
\Omega ;L_{per}^{\Phi }\left( Y\times Z\right) ^{d}\right) }^{\beta },
\end{align*}
for all $\left( f,\mathbf{f}\right) ,\left( f^{\prime },\mathbf{f}^{\prime
}\right) \in \mathcal{C}\left( \overline{\Omega };\mathcal{C}_{per}\left(
Y\times Z\right) ^{d+1}\right).$ 

\noindent Since $\mathcal{C}\left( \overline{\Omega }%
;\mathcal{C}_{per}\left( Y\times Z\right) ^{d+1}\right) $ is dense in $%
L^{\Phi }\left( \Omega ;L_{per}^{\Phi }\left( Y\times Z\right) ^{d+1}\right) 
$, the function $a$ extends by continuity to a mapping not relabeled from $L^{\Phi
}\left( \Omega ;L_{per}^{\Phi }\left( Y\times Z\right) ^{d+1}\right) $\ into 
$L^{\widetilde{\Phi }}\left( \Omega ;L_{per}^{\widetilde{\Phi }}\left(
Y\times Z\right) ^{d+1}\right) $\ such that all the inequalities in $\left(
ii\right) $ hold for every $u,v\in L^{\Phi }\left( \Omega
;L_{per}^{\Phi }\left( Y\times Z\right) \right) $ and every $\mathbf{U},%
\mathbf{V}\in L^{\Phi }\left( \Omega ;L_{per}^{\Phi }\left( Y\times Z\right)
^{d}\right) .$
\end{proof}

Thus, in the same spirit of \cite[Corollary 3.5]{Nnang Orlicz 2014}, we have the following result in the multiscale setting:
\begin{corollary}\label{cor3.2}
Under the same assumptions of Proposition \ref{prop3.3}, take $\left(
u_{\varepsilon }\right)_\varepsilon \subset L
^{\Phi }\left( \Omega \right) $ such that $u_{\varepsilon} \to u_0$
in
%EndExpansion
$L^{\Phi }\left( \Omega \right) $ as $ \varepsilon \rightarrow 0.$ Then,  for each $1\leq i\leq d$ and 
for every $\mathbf{f}\in \mathcal{C}\left( \overline{\Omega };\mathcal{C}_{per}\left( Y\times Z\right) ^{d}\right) $ we
have 
\begin{equation*}
a_{i}^{\varepsilon }\left( \cdot, \cdot,u_{\varepsilon },\mathbf{f}^{\varepsilon
}\right) \rightharpoonup a_{i}\left( \cdot, \cdot,u_{0},\mathbf{f}\right) \text{ weakly reiteratively two-scale}, \text{ in }L^{%
\widetilde{\Phi }}(\Omega), \text{as } \varepsilon \rightarrow 0.
\end{equation*}
\end{corollary}

\begin{proof}
Let $\left( u_{n}\right) _{n\in \mathbb{N}}\subset \mathcal{C}\left( \overline{\Omega }\right) $ be a sequence that
converges to $u_{0}$ in $L^{\Phi }\left( \Omega \right) $ as $n\rightarrow +\infty .$ Let $\varphi
\in L^{\Phi }\left( \Omega ;\mathcal{C}_{per}\left( Y\times Z\right) \right) 
$\ be arbitrarily fixed, let $1\leq i\leq d$, fix $\varepsilon>0,$\ and let
\begin{align}\label{Ie}
I\left( \varepsilon \right):=\int_{\Omega}a_{i}^{\varepsilon }\left( \cdot, \cdot,u_{\varepsilon },\mathbf{f}%
^{\varepsilon }\right) \varphi ^{\varepsilon }dx-\int_{\Omega
}\int_{Y}\int_{Z}a_{i}\left( \cdot, \cdot,u_{0},\mathbf{f}\right) \varphi dxdydz,
\end{align} 
the thesis will be achieved as $\lim_{\varepsilon \to 0}I(\varepsilon)=0$.
We can rewrite $I\left( \varepsilon \right) =\sum\limits_{k=1}^{4}I_{k}\left(
\varepsilon ,n\right) ,$ for any $n \in \mathbb N$, where 
\begin{align*}
I_{1}\left( \varepsilon ,n\right)
:=\int_{\Omega }\left( a_{i}^{\varepsilon }\left( \cdot, \cdot,u_{\varepsilon },%
\mathbf{f}^{\varepsilon }\right) -a_{i}^{\varepsilon }\left( \cdot, \cdot,u_{0},%
\mathbf{f}^{\varepsilon }\right) \right) \varphi ^{\varepsilon }dx,\\
I_{2}\left( \varepsilon ,n\right):=\int_{\Omega }\left( a_{i}^{\varepsilon
}\left( \cdot, \cdot,u_{0},\mathbf{f}^{\varepsilon }\right) -a_{i}^{\varepsilon
}\left( \cdot, \cdot,u_{n},\mathbf{f}^{\varepsilon }\right) \right) \varphi
^{\varepsilon }dx,\\
I_{3}\left( \varepsilon ,n\right):=\int_{\Omega }a_{i}^{\varepsilon }\left(
\cdot, \cdot,u_{n},\mathbf{f}^{\varepsilon }\right) \varphi ^{\varepsilon
}dx-\int_{\Omega }\int_{Y}\int_{Z}a_{i}\left( \cdot, \cdot,u_{n},\mathbf{f}\right)
\varphi dxdydz, \\
I_{4}\left( \varepsilon ,n\right):=\int_{\Omega }\int_{Y}\int_{Z}\left(
a_{i}\left( \cdot, \cdot,u_{n},\mathbf{f}\right) -a_{i}\left( \cdot, \cdot,u_{0},\mathbf{f}%
\right) \right) \varphi dxdydz.
\end{align*}

Applying Proposition \ref{prop3.3}, it follows that 
\begin{align}\left\vert I_{1}\left( \varepsilon
,n\right) \right\vert \leq 2c\left\Vert \varphi \right\Vert _{L^{\Phi
}\left( \Omega ;\mathcal{C}_{per}\left( Y\times Z\right) \right) }\left\Vert
u_{\varepsilon }-u_{0}\right\Vert _{\Phi ,\Omega }^{\alpha } \hbox{ for every }n \in \mathbb N, \label{I1}%\rightarrow
%0, \;\; \varepsilon \rightarrow 0\\
\\
\left\vert I_{2}\left( \varepsilon ,n\right) \right\vert \leq 2c\left\Vert
\varphi \right\Vert _{L^{\Phi }\left( \Omega ;\mathcal{C}_{per}\left(
Y\times Z\right) \right) }\left\Vert u_{n}-u_{0}\right\Vert _{\Phi ,\Omega
}^{\alpha } \hbox{ for every } \varepsilon >0. \label{I2}
%\rightarrow 0, \hbox{ as }n\rightarrow +\infty. 
\end{align}
 Hence $%\underset{
%n\rightarrow +\infty }
%{\lim }
\underset{\varepsilon \rightarrow 0}{\lim }
I_{1}\left( \varepsilon ,n\right)\equiv \lim_{\varepsilon \to 0} I_1(\varepsilon)=0$ while $
 \underset{n\rightarrow +\infty }{\lim }I_2\left( \varepsilon
,n\right) =0,$
uniformly with respect to $\varepsilon$.
\noindent Since $u_{n}\in \mathcal{C}\left( \overline{\Omega }\right)
\subset \mathcal{C}\left( \overline{\Omega };\mathcal{C}_{per}\left( Y\times
Z\right) \right) $ it follows again from Proposition \ref{prop3.3} that 
\begin{align}
	\label{I3}
\underset{ \varepsilon \rightarrow 0}%
{\lim }\int_{\Omega }a_{i}^{\varepsilon }\left( \cdot, \cdot,u_{n},\mathbf{f}%
^{\varepsilon }\right) \varphi ^{\varepsilon }dx=\int_{\Omega
}\int_{Y}\int_{Z}a_{i}\left( \cdot, \cdot,u_{n},\mathbf{f}\right) \varphi dxdydz,
\end{align}
and consequently that $\underset{n\rightarrow +\infty }{\lim }\underset{\varepsilon
\rightarrow 0}{\lim }I_3\left( \varepsilon ,n\right) =0.$ 
Finally 
\begin{align}\label{I4}
\left\vert I_{4}\left( \varepsilon ,n\right) \right\vert \leq 2\left\Vert
\varphi \right\Vert _{L_{per}^{\Phi }\left( \Omega \times Y\times Z\right)
}\left\Vert u_{n}-u_{0}\right\Vert _{\Phi,\Omega}^{\alpha },
\end{align} which again entails $\underset{n\rightarrow +\infty }{\lim }\underset{\varepsilon
\rightarrow 0}{\lim }I_4\left( \varepsilon ,n\right) =0.$
Hence, given $\delta >0$, we can find from \eqref{I1}, $\varepsilon_\delta>0$ such that $|I_1(\varepsilon)|<\frac{\delta}{4}$ for every $\varepsilon \leq\varepsilon_\delta.$ 
Moreover from \eqref{I2} we have that $|I_2(\varepsilon)|<\frac{\delta}{4}$ for every $n\geq n(\delta)$ uniformly for every $\varepsilon \leq \varepsilon_\delta$.  Then, from \eqref{I3} we can take any $n >n(\delta)$ and a $\varepsilon < \varepsilon_{2,\delta}$ (with $n > n_\delta$ and $\varepsilon_{2,\delta}\leq \varepsilon_\delta$, such that $|I_3| < \frac{\delta}{4}$. Finally  \eqref{I4} guarantees that for an $n > n_\delta$ and uniformly with respect to $\varepsilon$,  $|I_4| < \frac{\delta}{4}$, which concludes the proof.
\color{black}
\end{proof}

\begin{corollary}\label{cor3.4}
Under the assumptions of Proposition \ref{prop3.3}, consider $\left(
u_{\varepsilon}\right)_\varepsilon \subset L^{\Phi }\left( \Omega \right) $ such that $u_{\varepsilon
}\rightarrow u_{0}$ in $L^{\Phi }\left( \Omega \right) $ as $\varepsilon \rightarrow 0.$ Let $\phi _{\varepsilon
}\left( x\right): =\varphi _{0}\left( x\right) +\varepsilon \varphi
_{1}\left( x,\frac{x}{\varepsilon }\right) +\varepsilon
^{2}\varphi _{2}\left( x,\frac{x}{\varepsilon },\frac{x}{\varepsilon ^{2}}\right) ,$ $x\in \Omega ,\varphi _{0}\in \mathcal{D}\left(
\Omega \right) ,\varphi _{1}\in \mathcal{D}\left( \Omega \right) \otimes 
\mathcal{C}_{per}\left( Y\right) ,\varphi _{2}\in \mathcal{D}\left( \Omega
\right) \otimes \mathcal{C}_{per}\left( Y\right) \otimes \mathcal{C}%
_{per}\left( Z\right)$ ($
\phi _{\varepsilon }=\varphi _{0}+\varepsilon \varphi _{1}^{\varepsilon
}+\varepsilon ^{2}\varphi _{2}^{\varepsilon },$  shortly).  Then, 
\begin{align*}
a_{i}^{\varepsilon }\left( \cdot, \cdot,u_{\varepsilon },D\phi _{\varepsilon }\right)
\rightharpoonup a_{i}\left( \cdot, \cdot,u_{0},D\varphi _{0}+D_{y}\varphi
_{1}+D_{z}\varphi _{2}\right) \text{ in }L^{\widetilde{\Phi }}(\Omega)\\
\text{weakly reiteratively two-scale}
\end{align*}%
as $\varepsilon \rightarrow 0$, for any $1\leq i\leq d$.
Futhermore, given a sequence $\left(v_{\varepsilon }\right)_\varepsilon \subset L^{\Phi }\left( \Omega \right) ,$ such that $v_{\varepsilon }\rightharpoonup
v_{0}$ in $L^{\Phi }\left( \Omega \right)$ weakly reiteratively two-scale as $\varepsilon
\rightarrow 0,$ one has
\begin{align*}
\underset{\varepsilon \rightarrow 0}{\lim }\int_{\Omega
}a_{i}^{\varepsilon }\left( \cdot, \cdot,u_{\varepsilon },D\phi _{\varepsilon
}\right) v_{\varepsilon }dx = \\ 
\int_{\Omega }\int_{Y}\int_{Z}a_{i}\left( \cdot, \cdot,u_{0},D\varphi
_{0}+D_{y}\varphi _{1}+D_{z}\varphi _{2}\right) v_{0}dxdydz
\end{align*}
for any $1\leq i\leq d.$%
\end{corollary}
\begin{proof}
Given $\varphi \in L^{\Phi }\left( \Omega ;\mathcal{C}_{per}\left( Y\times
Z\right) \right) $ and $1\leq i\leq d,$ let 
\begin{equation*}
I^{\prime }\left( \varepsilon \right) :=\int_{\Omega }a_{i}^{\varepsilon
}\left( \cdot, \cdot,u_{\varepsilon },D\phi _{\varepsilon }\right) \varphi
^{\varepsilon }dx- \\ 
\int_{\Omega }\int_{Y}\int_{Z}a_{i}\left( \cdot, \cdot,u_{0},D\varphi
_{0}+D_{y}\varphi _{1}+D_{z}\varphi _{2}\right) \varphi dxdydz.
\end{equation*}
We can rewrite $I^{\prime }\left( \varepsilon \right) :=I\left( \varepsilon \right)
+J\left( \varepsilon \right),$ with 
\begin{equation*}
	J\left( \varepsilon \right)
:=\int_{\Omega }\left( a_{i}^{\varepsilon }\left( \cdot, \cdot,u_{\varepsilon },D\phi
_{\varepsilon }\right) -a_{i}^{\varepsilon }\left( \cdot, \cdot,u_{\varepsilon
},D\varphi _{0}+\left( D_{y}\varphi _{1}\right) ^{\varepsilon }+\left(
D_{z}\varphi _{2}\right) ^{\varepsilon }\right) \right) \varphi
^{\varepsilon }dx,
\end{equation*}
and $I(\varepsilon)$  as in \eqref{Ie} with $\mathbf{f}%
^{\varepsilon }$ replaced by $D\varphi _{0}+\left( D_{y}\varphi _{1}\right) ^{\varepsilon }+\left(
D_{z}\varphi _{2}\right) ^{\varepsilon }$. 

\noindent Thus, applying Corollary \ref{cor3.2}, $I\left( \varepsilon \right)
\rightarrow 0$ as $\varepsilon \rightarrow 0.$ 

\noindent Since $%
D\phi _{\varepsilon }\left( x\right) =D\varphi _{0}\left( x\right)
+D_{y}\varphi _{1}\left( x,\frac{x}{\varepsilon }\right) +D_{z}\varphi
_{2}\left( x,\frac{x}{\varepsilon },\frac{x}{\varepsilon ^{2}}\right)
+\varepsilon D_{x}\varphi _{1}\left( x,\frac{x}{\varepsilon }\right) $
$+\varepsilon D_{y}\varphi _{2}\left( x,\frac{x}{\varepsilon },\frac{x}{%
\varepsilon ^{2}}\right)$ $+\varepsilon ^{2}D_{x}\varphi _{2}\left( x,\frac{x}{%
\varepsilon },\frac{x}{\varepsilon ^{2}}\right),$ by condition $i)$ in Section \ref{O-Sspaces}
and Proposition \ref{propalphabeta}, we deduce  that \begin{equation*}\left\vert
J\left( \varepsilon \right) \right\vert \leq 2 c\left\Vert
\varphi \right\Vert _{L^{\Phi }\left( \Omega ;\mathcal{C}_{per}\left(
Y\times Z\right) \right) }
\left\Vert \varepsilon D_{x}\varphi
_{1}^{\varepsilon }+\varepsilon D_{y}\varphi _{2}^{\varepsilon }+\varepsilon
^{2}D_{x}\varphi _{2}^{\varepsilon }\right\Vert _{\Phi ,\Omega }^{\beta }
\end{equation*}

\noindent and $\left\vert J\left( \varepsilon \right) \right\vert \rightarrow 0$ as $%
\varepsilon \rightarrow 0,$ and this proves the first statement. The second one is obtained via similar
arguments to those in Corollary \ref{cor3.2}, exploiting the fact that there exists a  strong convergent sequence $(u_n) \subset C(\overline \Omega;C_{\rm per}(Y\times Z))$ to $u_0 \in L^\Phi(\Omega)$.
 
 Indeed one can define 
 \begin{align*}
 	I''(\varepsilon):=
\int_{\Omega
}a_{i}^{\varepsilon }\left( \cdot, \cdot,u_{\varepsilon },D\phi _{\varepsilon
}\right) v_{\varepsilon }dx -\\ 
\int_{\Omega }\int_{Y}\int_{Z}a_{i}\left( \cdot, \cdot,u_{0},D\varphi
_{0}+D_{y}\varphi _{1}+D_{z}\varphi _{2}\right) v_{0}dxdydz, 	
 \end{align*}
thus the proof will be concluded if we show that $\lim_{\varepsilon \to 0}I''(\varepsilon)=0$.
To this end we can rewrite $I''\left( \varepsilon \right) =\sum\limits_{k=1}^{5}I_{k}\left(
\varepsilon ,n\right) ,$ for any $n \in \mathbb N$, where 
\begin{align*}
	I_1(\varepsilon,n)(\equiv I_1(\varepsilon)):=\int_{\Omega
	}a_{i}^{\varepsilon }\left( \cdot, \cdot,u_{\varepsilon },D\phi _{\varepsilon
	}\right) v_{\varepsilon }dx - \\ 
\int_{\Omega
}a_{i}^{\varepsilon }\left( \cdot, \cdot,u_{\varepsilon },D\varphi_0+ (D_y\varphi_1)^\varepsilon+(D_z \varphi_2)^\varepsilon\right) v_{\varepsilon }dx
\\
I_2(\varepsilon,n)(\equiv I_2(\varepsilon)):=\int_{\Omega
}\left(a_{i}^{\varepsilon }\left( \cdot, \cdot,u_{\varepsilon },D\varphi_0+ (D_y\varphi_1)^\varepsilon+(D_z \varphi_2)^\varepsilon\right) \right. - \\ 
\left.a_{i}^{\varepsilon }\left( \cdot, \cdot,u_0,D\varphi_0+ (D_y\varphi_1)^\varepsilon+(D_z \varphi_2)^\varepsilon\right) v_{\varepsilon }\right)dx\\
I_3(\varepsilon, n):=\int_{\Omega
}a_{i}^{\varepsilon }\left( \cdot, \cdot,u_0,D\varphi_0+ (D_y\varphi_1)^\varepsilon+(D_z \varphi_2)^\varepsilon\right) v_{\varepsilon }dx-
\\
 \int_{\Omega
}a_{i}^{\varepsilon }\left( \cdot, \cdot,u_n,D\varphi_0+ (D_y\varphi_1)^\varepsilon+(D_z \varphi_2)^\varepsilon\right) v_{\varepsilon }dx\\
I_4(\varepsilon,n):=\int_{\Omega
}a_{i}^{\varepsilon }\left( \cdot, \cdot,u_n,D\varphi_0+ (D_y\varphi_1)^\varepsilon+(D_z \varphi_2)^\varepsilon\right) v_{\varepsilon }dx - \\
\int_{\Omega
}a_{i}\left( \cdot, \cdot,u_n,D\varphi_0+ (D_y\varphi_1)^\varepsilon+(D_z \varphi_2)^\varepsilon\right) v_0dx
\\
I_5(\varepsilon,n):=\int_{\Omega
}a_{i}\left( \cdot, \cdot,u_n,D\varphi_0+ D_y\varphi_1+D_z \varphi_2\right) v_0 dx-
\\
\int_{\Omega}a_{i}\left( \cdot, \cdot,u_0,D\varphi_0+ D_y\varphi_1+D_z \varphi_2\right) v_0dx.
\end{align*}
These integrals can be estimated as in Proposition \ref{prop3.3}, i.e. as in \cite{nnang reit}
\begin{align*}
	\lim_{\varepsilon \to 0}|I_1(\varepsilon)|=0, \hbox{uniformly with respect to } n\\
	\lim_{\varepsilon \to 0} |I_2(\varepsilon, n)|=
	\lim_{n\to +\infty}\lim_{\varepsilon \to 0}|I_2(\varepsilon)|=0\\
	\lim_{n\to +\infty}|I_3(\varepsilon,n )|=0 \hbox{ uniformly with respect to }\varepsilon,\\
\lim_{n\to +\infty}\lim_{\varepsilon \to 0}|I_4(\varepsilon, n)|=0 \hbox{ (as in the proof of \eqref{I3} above)}
\\
|I_5(\varepsilon, n)|\to 0 \hbox{ as } n\to +\infty, \hbox{uniformly with respect to } \varepsilon.
\end{align*}
These estimates conclude the proof.
\end{proof}

\subsection{Homogenization Result}\label{sub32}
 Following the notation in section \ref{O-Sspaces}, we set  $$\mathbb{F}_{0}^{1,\Phi }:=W_{0}^{1}L^{\Phi}\left( \Omega
\right) \times L^\Phi_{per}\left( \Omega ;W_{\#}^{1}L^{\Phi}\left( Y
\right) \right) \times L^\Phi\left( \Omega ;L_{per}^{\Phi}\left(
Y;W_{\#}^{1}L^{\Phi}\left( Z
\right) \right) \right) $$ and $$F^{\infty }:=\mathcal{D}\left( \Omega \right)
\times \left[ \mathcal{D}\left( \Omega \right) \otimes \mathcal{C}%
_{per}^{\infty }\left( Y
\right)\right] \times \left[ \mathcal{D}\left( \Omega \right) \otimes
\mathcal{C}_{per}^{\infty }(Y) \otimes \mathcal{C}_{per}^{\infty }\left( Z\right)\right].
$$
We endow $\mathbb{F}_{0}^{1,\Phi }$ by the norm 
\begin{equation*}
\left\Vert u\right\Vert _{\mathbb{F}_{0}^{1,\Phi }}:=\left\Vert
D_{x}u_{0}\right\Vert _{\Phi ,\Omega }+\left\Vert D_{y}u_{1}\right\Vert
_{\Phi ,\Omega \times Y}+\left\Vert D_{z}u_{2}\right\Vert _{\Phi ,\Omega
\times Y\times Z},
\end{equation*}
with $u=\left( u_{0},u_{1},u_{2}\right) \in \mathbb{F}%
_{0}^{1,\Phi }$. With this norm,\ $\mathbb{F}_{0}^{1,\Phi }$ is a Banach space. Moreover $%
F^{\infty }$\ is dense in $\mathbb{F}_{0}^{1,\Phi }$, where, in terms of notation, if  
\begin{align}
\left( \psi _{0},\psi _{1},\psi _{2}\right) \in
\mathcal{D}\left( \Omega \right) \times \left[ \mathcal{D}\left( \Omega \right) \otimes 
\mathcal{C}_{per}^{\infty }\left( Y\right) \right] \times \left[ \mathcal{D}\left( \Omega \right) \otimes 
\mathcal{C}_{per}^{\infty }\left( Y\right) \otimes \mathcal{C}_{per}^{\infty }\left( Z\right) \right], \label{3.57}
\end{align}
then $F^{\infty }=\left\{ \phi ,\phi :=\left( \psi _{0},\psi _{1},\psi
_{2}\right) \text{ in \eqref{3.57}}\right\}.$

\smallskip 
We are in position of proving the existence Theorem \ref{mainresult}, which we restate for the reader's convenience.

\medskip
{\bf Theorem \ref{mainresult}}
{\it 	Let \eqref{1.1} be the problem defined in Section \ref{sec1}, with $a$ and $f$ satisfying $(H_1)-(H_5)$. For each $\varepsilon
	>0$, let $u_{\varepsilon }$ be a solution of \eqref{1.1}. Then there exists a not relabeled subsequence and Then there exists a not relabeled subsequence  and $u:=\left( u_{0},u_{1},u_{2}\right)\in \mathbb{F}_{0}^{1,\Phi }:=W_{0}^{1}L^{\Phi}\left( \Omega
	\right) \times L^\Phi_{per}\left( \Omega ;W_{\#}^{1}L^{\Phi}\left( Y
	\right) \right) \times L^\Phi\left( \Omega ;L_{per}^{\Phi}\left(
	Y;W_{\#}^{1}L^{\Phi}\left( Z
	\right) \right) \right)$ such that \eqref{3.58}, \eqref{3.59} and \eqref{3.60} hold, namely
\begin{align*}
u_{\varepsilon }\rightharpoonup u_{0}\text{ in }W_{0}^{1}L^{\Phi }\left( \Omega
\right) -\text{weakly, \; and}
\\
D_{x_{i}}u_{\varepsilon }\rightharpoonup
D_{x_{i}}u_{0}+D_{y_{i}}u_{1}+D_{z_{i}}u_{2} \nonumber \\
\text{ weakly reiteratively two-scale in }L^{\Phi }(\Omega), 1\leq i\leq N.
\end{align*}
Moreover  the function $u:=\left( u_{0},u_{1},u_{2}\right)$ %\in \mathbb{F}_{0}^{1,\Phi } $ 
is a solution of
\begin{equation*}
\left\{ 
\begin{tabular}{l}
$\int_{\Omega }\int_{Y}\int_{Z}a\left(
y, z,u_{0},Du_{0}+D_{y}u_{1}+D_{z}u_{2}\right)\cdot \left(
Dv_{0}+D_{y}v_{1}+D_{z}v_{2}\right) dxdydz$ \\
\\ 
$=\int_{\Omega }fv_{0}dx,$ for all $v=\left( v_{0},v_{1},v_{2}\right) \in 
\mathbb{F}_{0}^{1,\Phi }$. \\ 
\end{tabular}%
\right. 
\end{equation*}
}
\color{black}
\begin{proof}
Observing that $\underset{0<\varepsilon \leq 1}{\sup }\left\Vert u_{\varepsilon
}\right\Vert _{W_{0}^{1}L^{\Phi }\left( \Omega
\right) }<\infty ,$ one can extract a not relabeled
subsequence such that \eqref{3.58} and \eqref{3.59} hold.
We show that $u:=\left( u_{0},u_{1},u_{2}\right) $ defined by \eqref{3.58} and 
\eqref{3.59} is a solution of \eqref{3.60}. 

To this end, for  $\phi :=\left( \psi
_{0},\psi _{1},\psi _{2}\right) \in F_{0}^{\infty },$ consider $\phi
_{\varepsilon }\left( x\right) =\psi _{0}\left( x\right) +\varepsilon \psi
_{1}\left( x,\frac{x}{\varepsilon }\right) +\varepsilon
^{2}\psi _{2}\left( x,\frac{x}{\varepsilon },\frac{x}{%
\varepsilon ^{2}}\right) $. We have 
\begin{align*}
\int_{\Omega }a^{\varepsilon }\left(
\cdot, \cdot,u_{\varepsilon },Du_{\varepsilon }\right) \cdot Du_{\varepsilon
}dx=\int_{\Omega }fu_{\varepsilon }dx;
\\
\int_{\Omega }a^{\varepsilon }\left( \cdot, \cdot,u_{\varepsilon },Du_{\varepsilon
}\right)  \cdot D\phi _{\varepsilon }dx=\int_{\Omega }f\phi _{\varepsilon }dx.
\end{align*}
Thus, by $(H_4)$,
\begin{align*}0\leq \int_{\Omega }\left( a^{\varepsilon }\left( \cdot, \cdot,u_{\varepsilon
},Du_{\varepsilon }\right) -a^{\varepsilon }\left( \cdot, \cdot,u_{\varepsilon
},D\phi _{\varepsilon }\right) \right) \cdot\left( Du_{\varepsilon }-D\phi
_{\varepsilon }\right) dx=\int_{\Omega }f\left( u_{\varepsilon }-\phi
_{\varepsilon }\right) dx
\\
-\int_{\Omega }a^{\varepsilon }\left( \cdot, \cdot, u_{\varepsilon}, D\phi
_{\varepsilon }\right) \cdot\left( Du_{\varepsilon }-D\phi _{\varepsilon
}\right) dx,
\end{align*}
with $f\in L^{\widetilde{\Phi }}\left( \Omega \right) \subset L^{%
\widetilde{\Phi }}\left( \Omega ;\mathcal{C}_{per}\left( Y\times Z\right)
\right)$ and
$\phi _{\varepsilon }\rightarrow \psi _{0}$ in $L^{\Phi }\left( \Omega
\right) $ (and so the convergence to $\psi_0$ is also in the strong reiterated two-scale convergence setting). 
Moreover, since $u_{\varepsilon }\rightharpoonup u_{0}$ in $W_{0}^{1}L^{\Phi }\left( \Omega \right)$, and the embedding $W_{0}^{1}L^{\Phi }\left( \Omega
\right) \hookrightarrow L^{\Phi }\left( \Omega\right) $ is compact, then $u_{\varepsilon }\rightarrow u_{0}$ in $L^{\Phi
}\left( \Omega\right) $ strongly and in the reiterated two-scale convergence setting. 

Therefore, 
\begin{align*}
\int_{\Omega }f\left( u_{\varepsilon }-\phi _{\varepsilon }\right)
dx\rightarrow \int_{\Omega }\int_{Y}\int_{Z}f\left( u_{0}-\psi _{0}\right)
dxdydz=\int_{\Omega }f\left( u_{0}-\psi _{0}\right) dx.
\end{align*}

By Corollary \ref{cor3.4}, it results that:
\begin{align*}
\int_{\Omega }a^{\varepsilon }\left( \cdot, \cdot, u
_{\varepsilon} ,D\phi _{\varepsilon }\right) \cdot \left( Du_{\varepsilon }-D\phi
_{\varepsilon }\right) dx\rightarrow\\
\int_{\Omega }\int_{Y}\int_{Z}a\left( y, z, u_{0},D\psi _{0}+D_{y}\psi
_{1}+D_{z}\psi _{2}\right)\cdot\\
\left( D\left( u_{0}-\psi _{0}\right) +D_{y}\left( u_{1}-\psi _{1}\right)
+D_{z}\left( u_{2}-\psi _{2}\right) \right) dxdydz,
\end{align*}
as $\varepsilon \to 0$.
Thus 
\begin{align*}
0\leq \int_{\Omega
}\int_{Y}\int_{Z}f\left( u_{0}-\psi_{0}\right) dxdydz-\int_{\Omega
}\int_{Y}\int_{Z}a\left( y, z, u_{0},D\psi _{0}+D_{y}\psi _{1}+D_{z}\psi
_{2}\right)\cdot
\\
\left( D\left( u_{0}-\psi _{0}\right) +D_{y}\left( u_{1}-\psi _{1}\right)
+D_{z}\left( u_{2}-\psi _{2}\right) \right) dxdydz.
\end{align*} 
Using density of $
F_{0}^{\infty }$ in $\mathbb{F}_{0}^{1,\Phi }$  the result remains valid
for all $\phi \in \mathbb{F}_{0}^{1,\Phi }$. 
In particular, choosing $\phi
:=u-tv,v=\left( v_{0},v_{1},v_{2}\right) \in \mathbb{F}_{0}^{1,\Phi }$, and dividing
 by $t>0$ we get:
\begin{align*}
0\leq \int_{\Omega }fv_{0}dx-\int_{\Omega }\int_{Y}\int_{Z}a\left(
y, z,u_{0},\mathbb{D}u-t\mathbb{D}v\right) \cdot \mathbb{D}vdxdydz,
\end{align*} where 
$\mathbb{D}v=Dv_{0}+D_{y}v_{1}+D_{z}v_{2}.$ 

Using the continuity of $a$ in its
 last argument and passing to the limit as $t\rightarrow 0,$ we are led
to 
\begin{align*}0\leq \int_{\Omega }fv_{0}dx-\int_{\Omega }\int_{Y}\int_{Z}a\left(
y, z,u_{0},\mathbb{D}u\right) \cdot \mathbb{D}vdxdydz,
\end{align*} that is 
%\begin{align*}\int_{\Omega
%}fv_{0}dx\geq \int_{\Omega }\int_{Y}\int_{Z}a_{i}\left( y, z,u_{0},\mathbb{D}%
%u\right) \cdot\mathbb{D}vdxdydz
%\end{align*} 
for all $v\in \mathbb{F}_{0}^{1,\Phi }.$
Thus, if we replace $%
v=\left( v_{0},v_{1},v_{2}\right) \in \mathbb{F}_{0}^{1,\Phi }$ by $%
v^{1}=\left( -v_{0},-v_{1},-v_{2}\right) $ we deduce:
\begin{align*}
-\int_{\Omega }fv_{0}dx\geq \int_{\Omega }\int_{Y}\int_{Z}a\left(
y, z,u_{0},\mathbb{D}u\right) \cdot\mathbb{D}v^1dxdydz, i.e.;
\\
-\int_{\Omega }fv_{0}dx\geq -\int_{\Omega }\int_{Y}\int_{Z}a\left(
y, z,u_{0},\mathbb{D}u\right) \cdot\mathbb{D}vdxdydz
\end{align*}
%or equivalently $%
%\int_{\Omega }fv_{0}dx\leq \int_{\Omega }\int_{Y}\int_{Z}a_{i}\left(
%\cdot, \cdot,u_{0},\mathbb{D}u\right) %.\mathbb{D}vdxdydz$ and 
thus the equality follows, i.e.
 $u=\left( u_{0},u_{1},u_{2}\right) $ verifies \eqref{3.60}. 
 \end{proof}

\begin{remark}\label{remuniq}
Assuming $(H_6)$ (which is a strict monotonicity assumption on $a$ and not very restrictive, for instance, in the case of $a$ not dependent on the third variable) one can
prove that \eqref{3.60} has a unique solution. Take $w=\left(
w_{0},w_{1},w_{2}\right) $ another solution of \eqref{3.60}, then
\begin{align*}
\int_{\Omega }fw_{0}dx=\int_{\Omega }\int_{Y}\int_{Z}a\left( y, z,u_{0},%
\mathbb{D}u\right) \cdot \mathbb{D}wdxdydz,
\\
\int_{\Omega }fu_{0}dx=\int_{\Omega }\int_{Y}\int_{Z}a\left(y, z,w_{0},%
\mathbb{D}w\right) \cdot \mathbb{D}udxdydz,
\\
-\int_{\Omega }fu_{0}dx=-\int_{\Omega }\int_{Y}\int_{Z}a\left(y, z,u_{0},%
\mathbb{D}u\right) \cdot \mathbb{D}udxdydz,
\\
-\int_{\Omega }fw_{0}dx=-\int_{\Omega }\int_{Y}\int_{Z}a\left(y, z,w_{0},%
\mathbb{D}w\right) \cdot \mathbb{D}wdxdydz,
\end{align*}
Consequently, by $(H_6)$
\begin{align*}
 0=\int_{\Omega}\int_{Y}\int_{Z}\left( a\left(y, z,w_{0},\mathbb{D}w\right) -a\left(
y, z,u_{0},\mathbb{D}u\right) \right) .\left( \mathbb{D}w-\mathbb{D}u\right)
dxdydz\geq \\
\gamma \int_{\Omega }\int_{Y}\int_{Z}\Phi \left( \left\vert \mathbb{D}w-
\mathbb{D}u\right\vert \right) dxdydz.
\end{align*}
Then by \eqref{1.2}, we can apply \cite[Lemma 2]{kenne Nnang}, in turn relying on \cite[Proposition 2.1]{MR}, which entail that the norm in $L^\Phi(\Omega \times Y \times Z)$ of $\mathbb D w- \mathbb D u$ is null. 
Hence, standard derivation and integration on $Z$ and $Y$, together with the fact that $u,w \in \mathbb F^{1, \phi}_0 $ guarantee that  $w=u$. We omit the details referring to \cite{nnang reit} where the same result is obtained in standard Sobolev setting.
\end{remark}

It is worth to point out that equation  \eqref{3.60} is referred as global
homogenization problem for \eqref{1.1} and is equivalent to the
following three systems (cf. also \cite[Theorem 2.11]{All2} and \cite{LNW} for the classical Sobolev setting and \cite{NN}):
\begin{align}
\int_{\Omega }\int_{Y}\int_{Z}a\left(
y, z,u_{0},Du_{0}+D_{y}u_{1}+D_{z}u_{2}\right) \cdot D_{z}v_{2}dxdydz=0 \nonumber \\ 
\hbox{for all }v_{2}\in L^\Phi\left( \Omega ;L_{per}^\Phi\left(
Y;W_{\#}^{1}L^\Phi\left( Z\right) \right) \right),
 \label{3.64}
\\
\int_{\Omega }\int_{Y}\left( \int_{Z}a\left(
y, z,u_{0},Du_{0}+D_{y}u_{1}+D_{z}u_{2}\right) dz\right) \cdot D_{y}v_{1}dxdy=0 \nonumber \\ 
\hbox{ for all }v_{1}\in L^\Phi_{\rm per}\left( \Omega ;W_{\#}^{1}L^\Phi\left( Y\right) \right), 
\label{3.65}
\\
\int_{\Omega }\left( \int_{Y}\int_{Z}a\left(y, z,u_{0},Du_{0}+D_{y}u_{1}+D_{z}u_{2}\right) dydz\right)\cdot
Dv_{0} dx \nonumber \\ 
=\int_{\Omega }fv_{0}dx \hbox{ for all }v_{0}\in W_{0}^{1}L^{\Phi }\left( \Omega
\right).
 \label{3.66}
\end{align}

Now we are in position to derive a macroscopic homogenized problem, as in \cite[Corollary 2.12]{All2}. Hence, let  $r \in \mathbb R$ and $\xi \in
\mathbb{R}
^{d}$ be arbitrarily fixed. For a.e.  $y \in Y$, consider the variational cell problem in \eqref{3.67}, whose solution is denoted by $\pi_2(y, r,\xi)$, i.e.
\begin{equation}\label{1.12}
\left\{ 
\begin{tabular}{l}
$\hbox{find } \pi _{2}\left(y,r, \xi \right) \in W_{\#}^{1}L^\Phi\left( Z\right) $ \hbox{such that} \\ 
$\int_{Z}a\left( y,z,r,\xi +D_{z}\pi _{2}\left( y,r,\xi \right) \right)
\cdot D_{z}\theta dz=0$ for all $\theta \in W_{\#}^{1}L^\Phi\left( Z\right).$%
\end{tabular}%
\right. 
\end{equation}%
\color{black}
This problem has a solution (arguing as in \cite[Theorem 3.2]{Y}), unique under assumption $(H_6)$,  as in the classical case (see \cite[page 18]{NN} which in turn relies on \cite{HS}).
% and (\ref{3.64})
%leads to $$\left\{ 
%\begin{tabular}{l}
%$\int_{Z}a\left( y,z,u_{0}\left( x\right) ,\underbrace{Du_{0}\left( x\right)
%+D_{y}u_{1}\left( x,y\right) }+D_{z}u_{2}\left( x,y,z\right) \right)
%\cdot D_{z}\theta dz=0$ \\ 
%for all $\theta \in W_{\#}^{1}L^\Phi\left( Z
%\right) $ and almost all $\left( x,y\right) \in %\Omega \times 
%TCIMACRO{\U{211d} }%
%BeginExpansion
%\mathbb{R}
%EndExpansion
%_{y}^{d}.$%
%\end{tabular}%
%\right. $$

%{\color{blue} It is worth to observe that the uniqueness of the solution should follow from a similar to the classical variational formulation....}
\noindent
Comparing \eqref{3.66} with (\ref{1.12}) for  $r= u_0(x)$ and $\xi =Du_{0}\left( x\right)
+D_{y}u_{1}\left( x,y\right) $ we can consider %for a.e. $x \in \Omega$ and $y \in Y$, $ u_{2}(x,y,z):=\pi _{2}\left(y, u_0(x), Du_{0}\left(
%x\right) +D_{y}u_{1}\left( x,y\right) %\right)(z),$ where 
$$
\Omega \times 
\mathbb{R}
_{y}^{d} \ni \left( x,y\right) \rightarrow \pi _{2}\left(y, u_0(x), Du_{0}\left( x\right)
+D_{y}u_{1}\left( x,y\right) \right)  \in W_{\#}^{1}L^\Phi\left( Z
\right), $$
%such that 
%$$\left( x,y\right) \rightarrow \pi _{2}\left(y, %u_0(x), Du_{0}\left( x\right)
%+D_{y}u_{1}\left( x,y\right) \right).
%$$
 %and this in particular ensures that $D_z\pi_2(D u_0(x)+ D_y u_1(x,y))(z)$ has the same measurability and differentiability properties of $D_z u_2$ with respect to $x,y$ and $z$.
Hence, defining  for a.e. $y \in Y$, and for any $(r, \xi) \in \mathbb R \times \mathbb R^d$, $h$ as in \eqref{h}, namely  
\begin{equation*}h\left(y,r, \xi \right) :=\int_{Z}a_{i}\left( y,z,r ,\xi
+D_{z}\pi _{2}\left(y,r, \xi \right) \right) dz,\end{equation*} \eqref{3.65} becomes \begin{align*}
\int_{\Omega }\int_{Y}h\left(y, u_0(x), Du_{0}(x)+D_{y}u_{1}(x,y)\right) \cdot D_{y}v_{1}dxdy=0, \end{align*}
 for all  $v_{1}\in L^{\phi}_{per}\left( \Omega ;W_{\#}^{1}L^\Phi\left( Y
\right) \right).$

\noindent Consequently, in analogy with the previous steps, one can consider, for any $(r,\xi) \in \mathbb R\times
\mathbb{R}
^{d},$ the function $\pi _{1}\left(r, \xi \right) \in
W_{\#}^{1}L^\Phi\left( Y
\right) $ solution of the variational problem in \eqref{3.69b},  (unique if $(H_6)$ exists)i.e.
\begin{equation*}
\left\{ 
\begin{tabular}{l}
\hbox{ find} $\pi _{1}\left(r, \xi \right) \in W_{\#}^{1}L^{\Phi}\left( Y
\right) $ \hbox{ such that } \\ 
$\int_{Y}h\left(y, r, \xi +D_{y}\pi _{1}\left(r, \xi \right) \right) \cdot D_{y}\theta
dy=0$ for all $\theta \in W_{\#}^{1}L^{\Phi}\left( Y
\right). $%
\end{tabular}%
\right.  \label{3.69}
\end{equation*}

Note also that (\ref{3.65}) lead us to $u_{1}=\pi _{1}\left(u_0,
Du_{0}\right) .$ Set, again, for $(r,\xi) \in \mathbb R\times \mathbb R^d$
\begin{align*}
q\left( r,\xi \right) =\int_{Y}h\left(y,r, \xi +D_{y}\pi
_{1}\left( r,\xi \right) \right) dy,
\end{align*}
the function $q$ in \eqref{q} is well defined.
 Moreover, it results from \eqref{3.66} and the above cell problems that 
\begin{align}
\int_{\Omega }q\left( u_0, Du_{0}\right) \cdot Dv_{0}dx=\int_{\Omega }f\cdot v_{0}dx\text{
for all }v_{0}\in W_{0}^{1}L^{\Phi }\left( \Omega \right) .  \label{3.71}
\end{align}

 Finally, analogously to the already mentioned standard $H^1$ setting in \cite{All2}, we have the following result, linking the macroscopic problem to the single iterated ones.

 {\bf Theorem \ref{maincor}} {\it For every $\varepsilon >0$, let \eqref{1.1} be  such that $a$ and $f$ satisfy $(H_1)-(H_6)$. 

Let $u_0 \in W_{0}^{1}L^{\Phi }(\Omega)$ be the solution  defined by means of \eqref{3.60}. Then, it is the unique solution of the
macroscopic homogenized problem 
\begin{equation}
-{\rm div} q\left( u_0, Du_{0}\right) =f\text{ in }\Omega ,u_{0}\in W_{0}^{1}L^{\Phi
}\left( \Omega \right),  \label{3.72}
\end{equation}
where $q$ is defined by \eqref{q}, taking into account \eqref{h}, \eqref{3.67} and \eqref{3.69b}.}
\begin{proof}%[Proof of Theorem \ref{maincor}]
From \eqref{3.71}, the function $u_{0}$ is a solution of \eqref{3.72}. Let $w_{0}\in W_{0}^{1}L^{\Phi }\left( \Omega \right) $ be another
solution of \eqref{3.72}, then we have
\begin{align*}-\int_{\Omega }q\left(u_0,
Du_{0}\right) \cdot Dw_{0}dx=-\int_{\Omega }f\cdot w_{0}dx;\\
\int_{\Omega }q\left(u_0,
Du_{0}\right) \cdot Du_{0}dx=\int_{\Omega }f\cdot u_{0}dx;\\
\int_{\Omega }q\left(w_0, Dw_{0}\right) \cdot Dw_{0}dx=\int_{\Omega
}f\cdot w_{0}dx;
\\-\int_{\Omega }q\left(w_0, Dw_{0}\right) \cdot Du_{0}dx=-\int_{\Omega
}f\cdot u_{0}dx.
\end{align*} 
Thus 
\begin{align*}
\int_{\Omega }\left( q\left(u_0, Du_{0}\right) -q\left(w_0,
Dw_{0}\right) \right) \cdot \left( Du_{0}-Dw_{0}\right) dx=0.
\end{align*} 
Replacing $q$ by \eqref{q} we get: 
\begin{align*}\int_{\Omega }\int_{Y}\left( h\left(y, u_0, Du_{0}+D_{y}\pi
_{1}\left(u_0, Du_{0}\right) \right) -h\left(y, w_0, Dw_{0}+D_{y}\pi _{1}\left(w_0,
Dw_{0}\right) \right) \right) \cdot \left( Du_{0}-Dw_{0}\right) dxdy=0, \hbox{ i.e. } \\
\int_{\Omega }\int_{Y}\int_{Z}\left[a\left( y,z,u_{0}\left( x\right) ,\underbrace{%
Du_{0}\left( x\right) +D_{y}\pi_{1}\left(u_0, Du_{0}\right) +D_{z}\pi
_{2}\left(y, u_0, Du_{0}+D_{y}\pi _{1}\left(u_0, Du_{0}\right) \right) }_{=P
}\right) \right.-\\
\left.a\left( y,z,w_{0}\left( x\right) ,\underbrace{Dw_{0}\left( x\right)
+D_{y}\pi _{1}\left(w_0, Dw_{0}\right) +D_{z}\pi _{2}\left(y, w_0, Dw_{0}+D_{y}\pi
_{1}\left(w_0, Dw_{0}\right) \right) }_{=F }\right) \right]\cdot \\
\left[ \left( Du_{0}-Dw_{0}\right) \right] dxdydz=0.
\end{align*}
Thus, \begin{align*}
	\int_{\Omega }\int_{Y}\int_{Z}\left( a\left( y,z,u_{0}\left(
x\right) ,P \right) -a\left( y,z,w_{0}\left( x\right) ,F \right)
\right) \cdot \left( P -F \right) dxdydz=0.
\end{align*} 
Since we have 
\begin{align*}0=\int_{\Omega
}\int_{Y}\int_{Z}\left( a\left( y,z,u_{0}\left( x\right) ,P \right)
\right) \cdot \left( P -Du_{0}\right) dxdydz=
\\
\int_{\Omega }\int_{Y}\int_{Z}\left( a\left( y,z,w_{0}\left( x\right) ,F
\right) \right) \cdot\left( F -Dw_{0}\right) dxdydz
\end{align*}
and 
\begin{align*}0=\int_{\Omega }\int_{Y}\int_{Z}\left( a\left( y,z,u_{0}\left( x\right)
,P \right) \right) \cdot \left( F -Dw_{0}\right) dxdydz=\\
\int_{\Omega }\int_{Y}\int_{Z}\left( a\left( y,z,w_{0}\left( x\right) ,F
\right) \right) \cdot \left( P -Du_{0}\right) dxdydz,
\end{align*}  similar arguments as
in Remark \ref{remuniq}, relying on $(H_6)$ give uniqueness.
\end{proof}
\color{black}

\section{Appendix}
In  the following, requiring that the coefficients $a$ in problem \eqref{1.1} satisfy $(H_1)-(H_4)$ together with the periodicity assumption  $(H_5),$ we prove which are the spaces of regular functions, where the compositions of functions as in the weak formulation of \eqref{1.1} are meaningful.
Consequently, the existence of solutions of \eqref{1.1} in the general case, presented in Section \ref{solweak}, will follow by density arguments. Regarding uniqueness we recall that a strict monotonicity assumption, such as $(H_6)$, would guarantee it, cf. Remark \ref{remuniq} for the uniqueness of solutions for the limiting problem.
\color{black}
\medskip

\noindent By \eqref{1.3} of $(H_2)$ applied to $\zeta ^{\prime }=0,\lambda ^{\prime }=\omega ,$ $\zeta
,\lambda $ and by $\left( H_{5}\right),$   for every $1\leq i\leq d,$  it follows

\begin{align*}\left\vert a_{i}\left( y,z,\zeta ,\lambda \right) \right\vert \leq
	\left\vert a_{i}\left( y,z,0,\omega \right) \right\vert +c_{1}\widetilde{%
		\Psi }^{-1}\left( \Phi \left( c_{2}\left\vert \zeta \right\vert \right)
	\right) +c_{3}\widetilde{\Phi }^{-1}\left( \Phi \left( c_{4}\left\vert
	\lambda \right\vert \right) \right) \leq c,
\end{align*}
for a.e. $y \in \mathbb R^d$ and every $(z, \zeta, \lambda) \in \mathbb R^d \times \mathbb R \times \mathbb R^d $.

Using $\left(
H_{1}\right)$, for a.e. $y\in \mathbb R^d$ , every $(\zeta, \lambda) \in \mathbb R \times \mathbb R^d$, for $1\leq i \leq d$,
$$a_{i}\left( y,\cdot,\zeta ,\lambda \right)
\in \mathcal{C}_b\left( 
\mathbb{R}
_z^d\right),$$ 
and for each $(z,\zeta,\lambda) \in \mathbb R^d\times \mathbb R \times \mathbb R^d$, $1\leq i\leq d$, $$a_{i}\left( \cdot,z,\zeta
,\lambda \right) \hbox{ is measurable and bounded.}$$ 

\noindent Consequently, with the notation of Section \ref{notations}, for $(\zeta, \lambda)\in \mathbb R \times \mathbb R^d$ and $1\leq i \leq d$, $$a_{i}\left(\cdot,z,\zeta ,\lambda \right) \in 
L_{per}^{\infty }\left( 
\mathbb{R}
_y^d;\mathcal{C}_b\left(
\mathbb{R}
_z^d\right) \right) \subset \mathfrak{X}_{per}^{\widetilde{\Phi }}\left( 
\mathbb{R}
_{y}^d;\mathcal{C}_b(\mathbb R^d_z)\right).
$$ 

\noindent Recalling Section \ref{notations},
traces (in the sense of \eqref{traceoperator}) are well defined on $L^{\infty }\left( 
\mathbb{R}
_{y}^{d};\mathcal{C}_b\left( 
\mathbb{R}
_{z}^{d}\right) \right) $, thus also on $\mathcal{C}\left( \overline{\Omega }%
;L^{\infty }\left( 
\mathbb{R}_y^N;\mathcal{C}_b\left( 
\mathbb{R}_z^N\right) \right) \right)$, in particular for $(u, Du)\in \mathcal C(\overline \Omega; \mathbb R^{d+1})$ (cf. \cite{nnang reit}).

\noindent Next, assume that, for every $1\leq i \leq d, (\zeta, \lambda) \in \mathbb R \times \mathbb R^d$, $a_{i}\left( \cdot, \cdot,\zeta ,\lambda \right) \in \mathcal{C}%
_{per}\left( Y\times Z\right).$ 
Thus, for every $\left( f,\mathbf{f}\right) \in \mathcal{C}_{per}\left(
Y\times Z\right) ^{d+1}$, it results that $a_{i}\left(
	\cdot, \cdot,f(\cdot,\cdot),\mathbf{f}(\cdot,\cdot)\right) \in \mathfrak{X}_{per}^{\widetilde{\Phi }}\left( 
	\mathbb{R}
	_y^d;\mathcal{C}_b(\mathbb R^d_z)\right), $ hence 	$a_{i}\left( \cdot, \cdot,f(\cdot,\cdot),\mathbf{f}(\cdot,\cdot)
	\right) \in \mathcal{C}_b\left( 
	\mathbb{R}_y^d\times 
	\mathbb{R}_z^d\right).$

Let  $\Lambda _{1}$ and $\Lambda_2$ be compact sets of $%
\mathbb{R}
$  and $\mathbb{R}
^{d}$, respectively  such that $\left( f\left( y,z\right) ,\mathbf{f}\left( y,z\right)
\right) $ $ \in \Lambda _{1}\times \Lambda _{2}$, for every $(y,z)\in Y\times Z.$
For every $1\leq i \leq d$, and denoting by $a_i$, also its
restriction (with respect to the two last arguments) on $\Lambda
_{1}\times \Lambda _{2}$, we have that  $a_{i}\in \mathcal{C}\left( \Lambda
_{1}\times \Lambda _{2};\mathcal{C}_{per}\left( Y\times Z\right) \right) .$

\noindent Then, suppose that, for every $1\leq i \leq d$, $a_{i}\left( \cdot, \cdot,\lambda _{1},\mathbf{\lambda }%
_{2}\right) :=\mathcal{X}\varphi$, with $\mathcal{X}\equiv \mathcal{X}(\lambda_1,\lambda_2)\in \mathcal {C}\left( \Lambda _{1}\times\Lambda _{2} \right)$, %;\mathcal{C}_{per}\left( Y\times Z\right) \right) ,$ 
and $\varphi \in \mathcal{C}_{per}\left( Y\times Z\right).$ Hence, there exists $\left(
f_{n}\right) _{n\in 
	\mathbb{N}}$, sequence of polynomials in the arguments $\left( \lambda _{1},\mathbf{\lambda }%
_{2}\right) $ such that $ f_{n}\rightarrow \mathcal{X}$ in $\mathcal{C}\left( \Lambda _{1}\times
\Lambda _{2}\right) $ as $n\rightarrow \infty.$

\noindent Therefore $f_{n}\left( f(\cdot,\cdot),%
\mathbf{f}(\cdot,\cdot)\right) \rightarrow \mathcal{X}\left( f(\cdot,\cdot),\mathbf{f}(\cdot,\cdot)\right) $ in $%
\mathcal{C}_b\left( 
\mathbb{R}
_{y}^{d}\times 
\mathbb{R}
_{z}^{d}\right) $ as $n\rightarrow \infty,$ i.e.  $$\underset{\left(
	y,z\right) }{\sup }\left\Vert f_n \left( f(y,z),\mathbf{f}\left(
y,z\right)\right)  - \mathcal{X}( f(y,z),\mathbf{f}(
y,z)) \right\Vert _{\infty }\rightarrow 0$$ and the function $\mathcal{X}%
\left( f,\mathbf{f}\right) \in \mathcal{C}_{per}\left( Y\times Z\right).$ Consequently, for every $1\leq i \leq d$, $$a_{i}\left( \cdot, \cdot,f(\cdot,\cdot),\mathbf{f}(\cdot,\cdot)\right) =\mathcal{X}\left( f(\cdot,\cdot),\mathbf{f}(\cdot,\cdot)
\right) \varphi(\cdot,\cdot) \in \mathcal{C}_{per}\left( Y\times Z\right).$$ Analogously for every $1\leq i \leq d$, if $a_{i}\left( \cdot, \cdot,\lambda _{1},\mathbf{\lambda }_{2}\right)
:=\sum\limits_{j=1}^s\mathcal{X}_{j}\left( \lambda _{1},\mathbf{%
	\lambda }_{2}\right) \varphi _{j}(\cdot,\cdot),$ %\left( \lambda _{1},\mathbf{\lambda }%
%_{2}\right) 
$ \,\left( \lambda _{1},\mathbf{\lambda }_{2}\right) \in \Lambda
_{1}\times \Lambda _{2}$, $s \in \mathbb N$, with ${\mathcal X}_i \in \mathcal {C}\left( \Lambda _{1}\times\Lambda _{2}\right) $, %;\mathcal{C}_{per}\left( Y\times Z\right) \right) ,$ 
and $\varphi_i \in \mathcal{C}_{per}\left( Y\times Z\right)$, then   $$ a_i \in \mathcal{C}(\Lambda_1\times \Lambda_2;\mathcal{C}_{per}\left( Y\times Z\right)).$$  

\noindent Taking $a_{i}$ arbitrarily in $\mathcal{C}\left( \Lambda _{1}\times
\Lambda _{2};\mathcal{C}_{per}\left( Y\times Z\right) \right) $ satisfying
the hypotheses $(H_1)-(H_5)$, by the density of $\mathcal{C}\left( \Lambda _{1}\times
\Lambda _{2}\right) \otimes $ $\mathcal{C}_{per}\left( Y\times Z\right) $ in 
$\mathcal{C}\left( \Lambda _{1}\times \Lambda _{2};\mathcal{C}_{per}\left(
Y\times Z\right) \right) $  it results that there exists $\left(
g_{n}\right) _{n\in 
	\mathbb{N}
}\subset \mathcal{C}\left( \Lambda _{1}\times \Lambda _{2}\right) \otimes $ $%
\mathcal{C}_{per}\left( Y\times Z\right) $ such that \begin{align*}\left\Vert
	g_{n}-a_{i}\right\Vert _{\mathcal{C}\left( \Lambda _{1}\times \Lambda _{2};%
		\mathcal{C}_{per}\left( Y\times Z\right) \right) }\equiv \\
	\underset{\begin{array}{ll}(
			\lambda _{1},\mathbf{\lambda }_{2}) \in \Lambda _{1}\times \Lambda _{2}%
			\\
			(y,z) \in \mathbb{R}_{y}^{d}\times 
			\mathbb{R}
			_{z}^{d}
	\end{array}}{\sup}
	\left\vert g_{n}\left( y,z,\lambda _{1},\mathbf{\lambda }%
	_{2}\right)  -a_{i}\left( y,z,\lambda _{1},\mathbf{\lambda }_{2}\right)
	\right\vert \rightarrow 0, \hbox{ as }n\rightarrow \infty .
\end{align*}

\noindent Then $g_{n}\left(
\cdot, \cdot,f(\cdot,\cdot),\mathbf{f}(\cdot,\cdot)\right) \rightarrow a_{i}\left( \cdot, \cdot,f(\cdot,\cdot ),\mathbf{f}(\cdot,\cdot )\right) $ in 
$\mathcal{C}_b\left( 
\mathbb{R}_{y}^{d}\times 
\mathbb{R}_{z}^{d}\right) $ as $n\rightarrow +\infty .$ Thus, since $g_{n}\left( \cdot, \cdot,f(\cdot,\cdot),
\mathbf{f}(\cdot.\cdot )\right)$ belongs to $ \mathcal{C}_{per}\left( Y\times Z\right) $, which is closed in $%
\mathcal{C}_b\left(\mathbb{R}_{y}^{d}\times 
\mathbb{R}_{z}^{d}\right),$ it results that $$a_{i}\left( \cdot, \cdot,f(\cdot,\cdot),\mathbf{f}(\cdot,\cdot)\right) \in \mathcal{C}%
_{per}\left( Y\times Z\right) \subset \mathfrak{X}_{per}^{\widetilde{\Phi }%
}\left( 
\mathbb{R}
_{y}^{d};\mathcal{C}_b(\mathbb R^d_z)\right).$$

\smallskip
\noindent For the general case, i.e. assuming that for every $1\leq i \leq d$, $a_{i}\left( \cdot, \cdot,\zeta ,\lambda \right) \in
L_{per}^{\infty }\left( 
\mathbb{R}_{y}^{d};\mathcal{C}_b\left( 
\mathbb{R}_{z}^{d}\right) \right) $ for each $\left( \zeta ,\lambda \right) \in 
\mathbb{R}
\times 
\mathbb{R}_{z}^{d}$, we exploit condition \eqref{3.84}. %\color{black} that $a_i$ ($i=1,\dots, d$) satisfy \eqref{3.84}. 

\smallskip
\noindent Let $\theta \in \mathcal{D}\left( 
\mathbb{R}^d\right) $ with $\theta \geq 0,$ supp \,$ \theta \subset \overline{B}_d(0,1)$
and $\int_{\mathbb R^d} \theta\left( y\right) dy=1.$ For any
integer $n\geq 1,$ set $\theta _{n}\left( y\right) :=n^d\theta \left(
ny\right), y\in \mathbb{R}
^{d}.$ Define, for every $\left( z,\zeta
,\lambda \right) \in 
\mathbb{R}^d\times \mathbb{R}\times 
\mathbb{R}^{d}$ and $1\leq i\leq d $,
\begin{align*}
	g_{n}^{i}\left( y,z,\zeta,\lambda \right) :=\int_{\mathbb R^d} \theta _{n}\left( \xi
	\right) a_{i}\left( y-\xi, z, \zeta,\lambda \right) d\xi,
\end{align*}%
and, let $g_{n}:=\left( g_{n}^{i}\right) _{1\leq i\leq d}$.
Clearly $g_{n}^{i}\in \mathcal{C}_{per}\left( Y\times Z\right) $ with $%
g_{n}\left( \cdot, \cdot,0,\omega \right) \in L^{\infty }\left( 
\mathbb{R}
_{y}^d\times
\mathbb{R}_z^d;
\mathbb{R}
\right)$. Moreover 
\begin{align*}
	\left( g_{n}\left( y,z,\zeta ,\lambda \right) -g_{n}\left(
	y,z,\zeta ^{\prime },\lambda ^{\prime }\right) \right) \cdot\left( \lambda
	-\lambda ^{\prime }\right) \geq c_{5}\Phi \left( \left\vert \lambda
	-\lambda'\right\vert \right),
	\hbox{ and }\\
	\left\vert g_{n}\left( y,z,\zeta ,\lambda \right) -g_{n}\left(
	y,z,\zeta ^{\prime },\lambda ^{\prime }\right) \right\vert \leq c_{1}%
	\widetilde{\Psi }^{-1}\left( \Phi \left( c_{2}\left\vert \zeta -\zeta
	^{\prime }\right\vert \right) \right) +c_{3}\widetilde{\Phi }^{-1}\left(
	\Phi \left( c_{4}\left\vert \lambda -\lambda \right\vert \right) \right),
\end{align*}
for all $y,z \in \mathbb R^d,\lambda ,\lambda ^{\prime }\in 
\mathbb{R}
^{d},\zeta ,\zeta ^{\prime }\in \mathbb{R}.$ From the above considerations, $g_{n}\left( \cdot, \cdot,f(\cdot,\cdot),\mathbf{f}(\cdot,\cdot)\right) \in \mathcal{C}%
_{per}\left( Y\times Z\right) ^{d}.$

\noindent Introducing the space
\begin{align*}\left( L^{\widetilde{\Phi }},l^{\infty }\right) \left( 
	\mathbb{R}_y^N;L^{\infty }\left( 
	\mathbb{R}
	_{z}^N\right) \right) :=\left\{ u\in L_{loc}^{\widetilde{\Phi }}\left( 
	\mathbb{R}
	_{y}^N;L^{\infty }\left( 
	\mathbb{R}
	_{z}^N\right) \right) :\left\Vert u\right\Vert _{\left( L^{\widetilde{\Phi 
		}},l^{\infty }\right) \left(
		\mathbb{R}
		_{y}^N;L^{\infty }\left( 
		\mathbb{R}
		_{z}^N\right) \right) }<\infty \right\}, 
\end{align*}
endowed with the norm
\begin{align*}
	\left\Vert u\right\Vert _{\left( L^{\widetilde{\Phi }},l^{\infty }\right)
		\left( 
		\mathbb{R}_{y}^N;L^{\infty }\left( 
		\mathbb{R}
		_{z}^N\right) \right) }=\underset{k\in 
		\mathbb{Z}
		^N}{\sup }\left\Vert \left\Vert u\left( y,\cdot\right) \right\Vert _{\infty
	}\right\Vert _{\widetilde{\Phi },k+Y},
\end{align*}
it results that, for every $1\leq i \leq d$,
$$a_{i}\left( \cdot, \cdot,f(\cdot,\cdot),\mathbf{f}(\cdot,\cdot)\right) \in L_{per}^{\infty }\left( 
\mathbb{R}
_{y}^N;\mathcal{C}_b\left(
\mathbb{R}_{z}^N\right) \right) \subset \left( 	L^{\widetilde{\Phi }},l^{\infty
}\right) \left( 
\mathbb{R}_y^N;L^{\infty }\left( 
\mathbb{R}_z^N\right) \right).$$
Hence, for $\eta >0$, there exists $N_{0}\in 
\mathbb{N}
,$ such that if $n\geq N_{0},$
\begin{align*}\left\vert g_{n}^{i}\left(
	y,z,\zeta ,\lambda \right) -a_{i}\left( y,z,\zeta ,\lambda\right)\right\vert \leq \\
	\int_{\frac{1}{n}B_{d}(0,1)} \left\vert \theta
	_{n}\left( \xi \right)a_{i}\left( y-\xi ,z,\zeta ,\lambda
	\right) -a_{i}\left( y ,z,\zeta ,\lambda \right) \right\vert d\xi \leq \eta.
\end{align*}
Furthermore,
\begin{align*}
	\left\Vert \left\Vert g_{n}^{i}\left( y,z,f(y,z),\mathbf{f}(y,z)\right)-a_{i}\left( y,z,f(y,z),\mathbf{f}(y,z)\right) \right\Vert _{L^\infty(Z) }\right\Vert _{%
		\widetilde{\Phi },Y}=
	\\
	\underset{\left\Vert u\right\Vert _{\Phi ,Y}\leq 1}{\sup }%
	\left\vert \int_{Y}\left\Vert g_{n}^{i}\left( y,z,f(y,z),\mathbf{f}(y,z)\right)
	-a_{i}\left( y,z,f(y,z),\mathbf{f}(y,z)\right) \right\Vert _{L^\infty(Z) }u(y)dy\right\vert.\nonumber
\end{align*}

Thus,
\begin{align*}\left\vert \int_{Y}\left\Vert g_{n}^{i}\left( y,z,f(y,z),\mathbf{f}(y,z)\right)
	-a_{i}\left( y,z,f(y,z),\mathbf{f}(y,z)\right) \right\Vert _{L^\infty(Z)}u(y)dy\right\vert =\\
	\left\vert \int_{Y}\underset{z}{\sup }\left\vert \int_{\frac{1}{n}%
		B_d(\omega,1)}\theta _{n}\left( \xi \right) a_{i}\left( y-\xi ,z,f(y,z),\mathbf{f}(y,z)\right)
	d\xi -a_{i}\left( y,z,f(y,z),\mathbf{f}(y,z)\right) \right\vert u\left( y\right)
	dy\right\vert= \\
	\left\vert \int_{Y}\underset{z}{\sup }\left\vert \int_{\frac{1}{n}%
		B_d(\omega,1)}\theta _{n}\left( \xi \right) \left( a_{i}\left( y-\xi ,z,f(y,z),\mathbf{f}(y,z)%
	\right) -a_{i}\left( y,z,f(y,z),\mathbf{f}(y,z)\right) \right) \right\vert d\xi
	u\left( y\right) dy\right\vert \leq \\
	\int_{Y}\underset{z}{\sup }\left\vert \int_{\frac{1}{n}B_d(\omega,1)}\theta
	_{n}\left( \xi \right) \left( a_{i}\left( y-\xi ,z,f(y,z),\mathbf{f}(y,z)\right)
	-a_{i}\left( y,z,f(y,z),\mathbf{f}(y,z)\right) \right) \right\vert \left\vert u\left(
	y\right) \right\vert d\xi dy
	\leq\\
	\int_{Y}\left\vert \int_{\frac{1}{n}B_d(\omega,1)}\theta _{n}\left( \xi
	\right) \eta \right\vert \left\vert u\left( y\right) \right\vert d\xi dy\leq
	\int_{\frac{1}{n}B_d(\omega,1)}\theta _{n}\left( \xi \right) \int_{Y}\left\vert \eta
	\right\vert \left\vert u\left( y\right) \right\vert dyd\xi \leq 
	C \eta \int_Y |u(y)|dy. 
	%\hbox{ replacing the line below}?   }
%\\\int_{\frac{1}{n}B_{d}(0,1)}\theta _{n}\left( \xi \right) 2\left\Vert
%u\right\Vert _{\Phi ,Y}\left\Vert \eta \right\Vert _{\widetilde{\Phi }%
	%,Y}d\xi \leq 2\left\Vert \eta \right\Vert _{\widetilde{\Phi },Y}\int_{\frac{1%
		%}{n}B_d(0,1)}\theta _{n}\left( \xi \right) d\xi %=2\left\Vert \eta \right\Vert _{%
	%\widetilde{\Phi },Y}
\end{align*}

%\color{magenta}

%So \eqref{estimate} will become $\left\Vert \left\Vert g_{n}^{i}\left( y,z,f(y,z),\mathbf{f}(y,z)\right)-a_{i}\left( y,z,f(y,z),\mathbf{f}(y,z)\right) \right\Vert _{L^\infty(Z) }\right\Vert _{%
%	\widetilde{\Phi },Y}\leq \eta?$

%\color{black}
\noindent Then, considering
%in the same order of ideas of \cite[(2.5)]{FNZOpuscula} {\color{blue} there the target space is $C_b$ here it was $L^\infty$, but for measurability issues, I guess we need $\mathcal C_b$ as well} we can introduce 
the space ${\Xi}^{\widetilde{\Phi }}\left( 
\mathbb{R}_{y}^N;{\mathcal C}_b\left( 
\mathbb{R}
_{z}^N\right) \right)$ in \eqref{Xi}, endowed with the norm \eqref{ornLBPer} (with the $N$-function $B$ therein replaced by $\widetilde{\Phi}$), taking into account Lemma \ref{lemma2.2} and Proposition \ref{prop2.1}, (i.e. 
% (i.e. $\left\Vert \cdot\right\Vert_{\Xi ^{\widetilde{\Phi}}\left( 
%	\mathbb{R}
%	_{y}^{N};\mathcal C_b\left( 
%	\mathbb{R}_{z}^{N}\right) \right)}$
%	with $B$ replaced by $\widetilde \Psi$).
%\begin{align*}
%	{\Xi}^{\widetilde{\Phi }}\left( 
%	\mathbb{R}_{y}^N;{\color{blue}{\mathcal C}_b}\left( 
%	\mathbb{R}
%	_{z}^N\right) \right) =\left\{ u\in L_{loc}^{\widetilde{\Phi }}\left( 
%	\mathbb{R}
%	_y^N;{\color{blue}{\mathcal C}_b}\left( 
%	\mathbb{R}
%	_{z}^N\right) \right) :\left\Vert
%	u\right\Vert _{\mathbb{E}^{\widetilde{\Phi },\infty }}<\infty \right\},
%\end{align*}
%\underset{%
%	0<\varepsilon \leq 1}{\sup }\left\Vert \left\Vert u\left( \frac{x}{%
%	\varepsilon },\cdot\right) \right\Vert _{\infty }\right\Vert _{\widetilde{\Phi }%
%	,B_N(0,1)},$
%where $\|\cdot\|_{\widetilde \phi, B_N(0,1)}$ is the norm of $L^{\widetilde \Phi}(B_N(0,1))$. 
%\bigskip
%\noindent On the other hand, as shown in \cite{FNZOpuscula}, starting from regular functions, 
%there exists $c>0$ such that $%
%\left\Vert \left\Vert u\left( \frac{x}{\varepsilon },\cdot\right) \right\Vert
%_{\infty }\right\Vert _{\widetilde{\Phi },B_d(0,1)}
%\leq c\left\Vert u\right\Vert _{\widetilde{\Phi %},Y\times Z},$ for any $0 <\varepsilon \leq 1$,
%\noindent and
$\left\Vert u\right\Vert _{\Xi ^{\widetilde{\Phi}}\left( 
\mathbb{R}
_{y}^d;\mathcal C_b\left( 
\mathbb{R}_{z}^d\right) \right)}$ equivalent to
$\left\Vert u\right\Vert _{\widetilde{\Phi },Y\times Z}$ ), the above estimates provide
%\bigskip With the above definitions there exists $c>0,$ such that $\left\Vert u\right\Vert _{\mathbb{%
	%E}^{\widetilde{\Phi },\infty }}\leq c\left\Vert u\right\Vert _{\widetilde{%
	%\Phi },Y};$ 
(with $ u\in \mathfrak{X}_{per}^{\widetilde{\Phi }}\left( 
\mathbb{R}
_{y}^d;\mathcal{C}_b(\mathbb R_z^d)\right) $) \begin{align*}\left\Vert g_{n}^{i}\left( \cdot, \cdot,f(\cdot,\cdot),\mathbf{%
	f}(\cdot,\cdot)\right) -a^{i}\left( \cdot, \cdot,f(\cdot,\cdot),\mathbf{f}(\cdot,\cdot)\right) \right\Vert _{\Xi ^{\widetilde{\Phi}}\left( 
	\mathbb{R}
	_{y}^d;\mathcal C_b\left( 
	\mathbb{R}_{z}^d\right) \right)}\leq 2c\left\Vert \eta \right\Vert _{\widetilde{%
		\Phi },Y}\end{align*} for every $n\geq N_{0}$.
The arbitrariness of $\eta $ proves that, for every $1\leq i \leq d$,
\begin{align*}
g_{n}^{i}\left( \cdot, \cdot,f(\cdot,\cdot),\mathbf{f}(\cdot,\cdot)\right) \rightarrow a_{i}\left(
\cdot, \cdot,f(\cdot,\cdot),\mathbf{f}(\cdot,\cdot)\right) \hbox{ in } {\Xi ^{\widetilde{\Phi}}\left( 
	\mathbb{R}
	_{y}^d;\mathcal C_b\left( 
	\mathbb{R}_{z}^d\right) \right)},
\end{align*}
and, since $g_{n}^{i}\left( \cdot, \cdot,f(\cdot,\cdot),\mathbf{f}(\cdot,\cdot)\right) \in \mathcal{C}_{per}\left( Y\times
Z\right)$, it results that $$a_{i}\left( \cdot, \cdot,f(\cdot,\cdot),\mathbf{f}(\cdot,\cdot)\right) \in \mathfrak{X}_{per}^{%
\widetilde{\Phi }}\left( 
\mathbb{R}
_{y}^d;\mathcal{C}_b(\mathbb R^d_z)\right).$$

%\end{claim}

%\end{claim}

\subsection{Definition of $a^{\protect\varepsilon }\left( \cdot, \cdot,w^{\protect%
	\varepsilon },\mathbf{W}^{\protect\varepsilon }\right) $ for $\left( w,%
\mathbf{W}\right) \in \mathcal{C}\left( \overline{\Omega };\mathcal{C}_b\left( 
\mathbb{R}_{y}^{d}\times 
\mathbb{R}_{z}^{d}\right) \right) $} \hfill

\bigskip

\noindent Let $U$ be a bounded open set of $\mathbb{R}
_{y}^d$ and let $u\in \mathcal{C}_b\left( \overline{U}\right) \otimes
L^{\infty }\left(
\mathbb{R}_t^d,F\right)$, where $F$ is a Banach space. It is well known that there
exists a negligible set $\mathcal{N\subset }
\mathbb{R}_{t}^d$ such that $\left\Vert u\left( y,t\right) \right\Vert _{F}\leq 
\underset{y\in \overline{U}}{\sup }\left\Vert u\left( y,\cdot\right) \right\Vert
_{L^{\infty }\left( 
\mathbb{R}
_{t}^d,F\right) },$ for every $y\in \overline{U}$ and $t\in 
\mathbb{R}_{t}^d\setminus \mathcal{N}.$
%were $u\left( z\right) =u\left( z,.\right) $
Moreover, given $s \in \mathbb N$, and assuming that $u=\sum_{i=1}^s\varphi _{i}\otimes \psi
_{i}$, with $\varphi _{i}\in \mathcal{C}\left( \overline{U}\right) ,\psi _{i}\in
L^{\infty }\left( 
\mathbb{R}
_{t}^d,F\right),$ one can define ${\widetilde u^{\varepsilon }}\in L^{\infty }\left( U,F\right) $ by $${\widetilde u^{\varepsilon }}\left( y\right):=\sum_{i=1}^s
\varphi _{i}\left( y\right) \otimes \psi _{i}\left( \frac{y}{\varepsilon }%
\right), \, y\in U.$$ 
Consequently, it is defined a linear trace operator $u\rightarrow
{\widetilde u^{\varepsilon }}$ from $\mathcal{C}_b\left( \overline{U}\right) \otimes
L^{\infty }\left( 
\mathbb{R}
_{t}^d,F\right) $ to $L^{\infty }\left( U,F\right) .$ Since $\mathcal{C}_b%
\left( \overline{U}\right) \otimes L^{\infty }\left( 
\mathbb{R}
_{t}^d,F\right) $\ is dense in $\mathcal{C}_b\left( \overline{U};L^{\infty
}\left( 
\mathbb{R}
_{t}^d,F\right) \right) $, it is also known that the trace operator $%
u\rightarrow {\widetilde u^{\varepsilon }}$ extends by continuity to a unique linear
operator from $\mathcal{C}_b\left( \overline{U};L^{\infty }\left( 
\mathbb{R}
_{t}^d,F\right) \right) $ to $L^{\infty }\left( U,F\right) $ still denoted by
$u\rightarrow {\widetilde u^{\varepsilon }}$ such that

\bigskip\ $\left\Vert {\widetilde u^{\varepsilon }}\right\Vert _{L^{\infty }\left(
U,F\right) }\leq \underset{y\in \overline{U}}{\sup }\left\Vert u\left(
y, \cdot\right) \right\Vert _{L^{\infty }\left( 
\mathbb{R}
_{t}^d,F\right) },$ for all $u\in $\ $\mathcal{C}_b\left( \overline{U}%
;L^{\infty }\left( 
\mathbb{R}
_{t}^d,F\right) \right) .$

\noindent Observe that, with the above strategy,  we have defined the application $u\rightarrow {\widetilde u^{\varepsilon }}$ which is the trace operator
on $\left\{ \left( y,t\right) :t=\frac{y}{\varepsilon },y\in U\right\}.$

Let $V$ be an open bounded subset of $%
\mathbb{R}_{z}^d$ and consider $\psi$ of the form $\psi =\varphi \otimes v\otimes
\phi $ with $\varphi \in \mathcal{C}\left( \overline{U}\right) ,v\in 
\mathcal{C}\left( \overline{V}\right) $ and $\phi \in L^{\infty }\left( 
\mathbb{R}_t^d;\mathcal{C}_b\left( 
\mathbb{R}
_{\zeta }^d\right) \right).$ For fixed $y$ and $t$ in $%
\overline{U}$ and $\mathbb{R}_t^d$, respectively, the function $\left( z,\zeta \right) \rightarrow \varphi \left(
y\right) v\left( z\right) \phi \left( t,\zeta \right) $ is an element of $%
\mathcal{C}_b\left( \overline{V}\times 
\mathbb{R}
_{\zeta }^d\right)$ such that its trace of order $\varepsilon $ on $%
\left\{ \left( z,\zeta \right) :\zeta =\frac{z}{\varepsilon },z\in V\right\}
$ is well defined. In particular the function $z\rightarrow \varphi \left(
y\right) v\left( z\right) \phi \left( t,z\right) $ belongs to $\mathcal{C}
\left( \overline{V}\right).$ Therefore, the function $\left( y,t\right)
\rightarrow \varphi \left( y\right) \phi \left( t,\cdot \right) $ belongs to $%
\mathcal{C}\left( \overline{U}\right) \otimes L^{\infty }\left( 
\mathbb{R}
_{t}^N;\mathcal{C}\left( \overline{V}\right) \right) .$
Thus, its trace of
order $\varepsilon >0$ on $\left\{ \left( y,t\right) :t=\frac{y}{\varepsilon 
}, y\in U \right\} $ is defined as above. It follows that the
function $\left( y,z\right) \rightarrow \varphi \left( y\right) v\left(
z\right) \phi \left( y,z\right) $ is well defined and belongs to $L^{\infty
}\left(U;\mathcal{C}\left( \overline{V}\right) \right) $ with 
\begin{equation}\label{4.1}
\left\vert \psi \left( y,z,y,z\right) \right\vert \leq \underset{a\in 
	\overline{U}}{\sup }\underset{b\in \overline{V}}{\sup }\left\Vert \psi
\left( a,\cdot, b, \cdot \right) \right\Vert _{L^{\infty }\left( 
	\mathbb{R}
	^d,\mathcal{C}_b\left(
	\mathbb{R}^d\right) \right) }.
\end{equation}%
Then for every positive integer $s$, let $\psi :=\sum\limits_{i=1}^s\varphi _{i}\otimes
v_{i}\otimes \phi _{i},$ where $\varphi _{i}\in \mathcal{C}\left( \overline{U}%
\right) ,v_{i}\in \mathcal{C}\left( \overline{V}\right) $ and $\phi _{i}\in
L^{\infty }\left(\mathbb{R}_{t}^{d},
\mathbb{R}_{z}^{d}\right) .$ From  the above formulas, it follows that each function $\left( y,z\right) \rightarrow
\varphi _{i}\left( y\right) v_{i}\left( z\right) \phi _{i}\left( y,z\right) $
belongs to $L^{\infty }\left( U;\mathcal{C}\left( \overline{V}\right)
\right).$ Consequently $\left( y,z\right) \rightarrow \psi \left(
y,z,y,z\right) $ is an element of $L^{\infty }\left( U;\mathcal{C}\left( 
\overline{V}\right) \right) $ verifying  \eqref{4.1}
%\begin{equation}
%\left\vert \psi \left( y,z,y,z\right) \right\vert \leq \underset{a\in 
%\overline{U}}{\sup }\underset{b\in \overline{V}}{\sup }\left\Vert \psi
%\left( a,\cdot,b,\cdot\right) \right\Vert _{L^{\infty }\left( 
%\mathbb{R}^{d};\mathcal{C}_b\left( 
%\mathbb{R}
%^{d}\right) \right) }  \label{3.13}
%\end{equation}%
Thus the restriction on $\mathcal{C}\left( \overline{U}\right) \otimes 
\mathcal{C}\left( \overline{V}\right) \otimes L^{\infty }\left( \mathbb{R}
_{t}^d;\mathcal{C}_b\left( 
\mathbb{R}
_{\zeta }^d\right) \right) $ of the transformation $\psi \rightarrow {\widetilde\psi
^{\varepsilon =1}}$ from $\mathcal{C}_b\left( \overline{U}\times \overline{V}%
;L^{\infty }\left( 
\mathbb{R}
_{t}^d;\mathcal{C}_b\left( 
\mathbb{R}_{\zeta }^d\right) \right) \right) $ to $L^{\infty }\left( U\times
V\right) $ is a continuous linear operator from $\mathcal{C}\left( \overline{U%
}\right) \otimes \mathcal{C}\left( \overline{V}\right) \otimes L^{\infty
}\left(
\mathbb{R}
_{t}^d;\mathcal{C}_b\left(
\mathbb{R}
_{\zeta }^d\right) \right) $ to $L^{\infty }\left( U;\mathcal{C}\left( 
\overline{V}\right) \right) .$ This operator, in turn, extends uniquely by continuity
and density to a linear and continuous operator from $\mathcal{C}\left( 
\overline{U}\times \overline{V};L^{\infty }\left( 
\mathbb{R}_{t}^d;\mathcal{C}_b\left( 
\mathbb{R}
_{\zeta }^d\right) \right) \right) $ to $L^{\infty }\left( U;\mathcal{C}
\left( \overline{V}\right) \right) $ still denoted by $\psi \rightarrow {\widetilde\psi
^{\varepsilon =1}}$ with 
\begin{equation*}
\left\Vert {\widetilde\psi ^{\varepsilon =1}}\right\Vert _{L^{\infty }\left( U;{\mathcal C
	}\left( \overline{V}\right) \right) }\leq \underset{a\in \overline{U}}{\sup }\;
\underset{b\in \overline{V}}{\sup }\left\Vert \psi \left( a,\cdot, b, \cdot \right)
\right\Vert _{L^{\infty }\left( \mathbb{R}
	^d,\mathcal{C}_b\left( 
	\mathbb{R}
	^d\right) \right) }
\end{equation*}%
for all $\psi \in \mathcal{C}_b\left( \overline{U}\times \overline{V}%
;L^{\infty }\left( 
\mathbb{R}
_{t}^d,\mathcal{C}_b\left( 
\mathbb{R}
_{\zeta }^d\right) \right) \right) .$ 

\noindent Moreover, given $\psi \in \mathcal{C}_b\left( 
\overline{U}\times \overline{V};L^{\infty }\left(
\mathbb{R}
_{t}^d,\mathcal{C}_b\left( 
\mathbb{R}
_{\zeta }^d\right) \right) \right)$, if for  any fixed $\left(
y,z\right) \in \overline{U}\times \overline{V}$ it results $\psi \left(
y,z,t,\zeta \right) \geq 0$ for every $\zeta \in 
\mathbb{R}
^d$ and for a.e. $t\in
\mathbb{R}
^d,$ then there exists a negligible set $\mathcal{N  \subset }\mathbb{R}
_{y}^d$ such that ${\widetilde\psi ^{\varepsilon =1}}\geq 0$ for all $\left(
y,z\right) \in \left( U\setminus \mathcal{N}\right) \times V.$ 

\noindent All the above considerations can be
generalized on $\mathcal{C}_b\left( 
\mathbb{R}_{y}^d\times 
\mathbb{R}_{z}^d;L^{\infty }\left( 
\mathbb{R}_{t}^d,\mathcal{C}_b\left( 
\mathbb{R}_{\zeta }^d\right) \right) \right),$ indeed, we have

\begin{proposition}\label{prop 3.1}
For $\psi \in \mathcal{C}_b\left(
\mathbb{R}
_{y}^d\times 
\mathbb{R}
_{z}^d;L^{\infty }\left(
\mathbb{R}_{t}^d;\mathcal{C}_b\left( 
\mathbb{R}
_{\zeta }^d\right) \right) \right),$ the trace of order $\varepsilon =1$ 
on  \hfill \eject $\left\{ \left( y,z,t,\zeta \right) :t=y\text{ and }\zeta =z,\left(
y,z\in 
\mathbb{R}
^d\right) \right\} $ of $\psi $ belongs to $L^{\infty }\left( 
\mathbb{R}_{y}^d;\mathcal{C}_b\left( 
\mathbb{R}
_{z}^d\right) \right).$ Moreover, for every $\left( y,z\right)\in 
\mathbb{R}^d,$ if $\psi \left( y,z,t,\zeta \right) \geq 0,$ for every $\zeta \in \mathbb{R}^d$ and for a.e. $t\in 
\mathbb{R}^d,$ then ${\widetilde\psi ^{\varepsilon =1}}\left( y,z\right) \geq 0,$ for
every $z\in \mathbb{R}^d$ and for a.e. $y\in 
\mathbb{R}^d.$
\begin{proof}
	Let $n$ be a positive integer. Let $U_{n}$ and $V_n$ be the open
	balls of $\mathbb{R}_y^d$  and $
	\mathbb{R}_z^d$, respectively, centered at $\omega $, with radius $n.$ Set $\psi_n:=\psi\lfloor_{\overline{U}_{n}\times \overline{V}_{n}},$ (i.e. the restriction with respect to the first two variables) then $\psi _{n}$ is an
	element of $\mathcal{C}_b\left( \overline{U}_{n}\times \overline{V_n};\right.$ $ \left. L^{\infty }\left( 
	\mathbb{R}_{t}^{d},\mathcal{C}_b\left( 
	\mathbb{R}
	_{\zeta }^{d}\right) \right) \right).$ Thus, following the  above arguments, one can define ${\widetilde\psi
		_{n}^{\varepsilon =1}}\in L^{\infty }\left( U_{n};\mathcal{C}\left( \overline{%
		V}_{n}\right) \right).$  Hence it results  that the sequence $\left({\widetilde\psi _{n}^{\varepsilon
			=1}}\right)_n$ verifies ${\widetilde\psi _{n}^{\varepsilon =1}}={\widetilde\psi
		_{n+1}^{\varepsilon =1}}\lfloor_{\overline{U}_{n}\times \overline{V}%
		_{n}}$, for every $n \in \mathbb N$. In particular this guarantees that there is a unique ${\widetilde \psi ^{\varepsilon =1}}\in
	L^{\infty }\left( 
	\mathbb{R}
	_{y}^{d},\mathcal{C}_b\left( 
	\mathbb{R}
	_{z}^{d}\right) \right) $ such that ${\widetilde\psi ^{\varepsilon =1}}\lfloor_{
		U_{n}\times V_{n}}= {\widetilde\psi _{n}^{\varepsilon =1}},$ for $n\in 
	\mathbb{N}.$ Moreover the continuity is a consequence of \eqref{4.1}. 
	
	\medskip
	To prove the second part, we start assuming that for every $y,z\in 
	\mathbb{R}^{d}$ it results $\psi \left( y,z,t,\zeta \right) \geq 0$ for every $\zeta \in 
	\mathbb{R}^{d}$ and for a.e. $t\in 
	\mathbb{R}^{d}.$ Then, for every $n \in \mathbb N$, $n >0$, it results that $\psi _{n}\left( y,z,t,\zeta \right) \geq 0$ for every $\zeta \in 
	\mathbb{R}^{d}$ and a.e. $t\in \mathbb{R}^{d}$. 
	%if $y\in \overline{U}_{n}$ and $z\in \overline{V}_{n}.$
	
	By the above considerations about traces (prior to Proposition \ref{prop 3.1}), for every $n \in \mathbb N$,
	there exists a negligible set $\mathcal{N}_{n}\subset 
	\mathbb{R}_{y}^{d}$ such that ${\widetilde\psi_{n}^{\varepsilon =1}}\left( y,z\right) \geq 0$
	for all $(y,z) \in \left(U_{n}\setminus \mathcal{N}\right)
	\times V_{n}.$ Then ${\widetilde \psi _{n}^{\varepsilon =1}}\left( y,z\right) \geq 0$ for
	every $\left( y,z\right) \in \left( U_{n}\setminus \mathcal{N}\right) \times
	V_{n}$, $ n\in \mathbb{N},$ where ${\mathcal N}=\bigcup\limits_{n\geq 1}\mathcal{N}%
	_{n},$ from which it follows that ${\widetilde\psi ^{\varepsilon =1}}\left( y,z\right) \geq 0$ for
	every $z\in \mathbb{R}^{d}$ and a.e. $y\in\mathbb{R}^{d}.$
\end{proof}
\end{proposition}

\begin{corollary}\label{cor42}
\bigskip Let $a:= (a_i)_{1\leq i \leq d}:\mathbb R^d \times\mathbb R^d \times \mathbb R \times \mathbb R^d\to \mathbb R^d $ satisfy $(H_1)-(H_4)$ and let $\left( w,\mathbf{W}\right) \in \mathcal{C}_b\left( 
\mathbb{R}_{y}^{d}\times 
\mathbb{R}
_{z}^{d}\right)^{d+1}.$ For $1\leq i\leq d,$ the function $\left(
y,z\right) \rightarrow a_{i}\left( y,z,w\left( y,z\right) ,\mathbf{W}\left(
y,z\right) \right) $ from $\mathbb{R}_{y}^{d}\times 
\mathbb{R}_{z}^{d}$ into $\mathbb{R}$ is an element of $L^{\infty }\left( 
\mathbb{R}_{y}^{d},\mathcal{C}_b\left( 
\mathbb{R}
_{z}^{d}\right) \right) ,$ denoted as $a_{i}\left( \cdot, \cdot,w,\mathbf{W}\right) $.

For every $\left( w,\mathbf{W}\right) ,\left( v,\mathbf{V}\right)
\in \mathcal{C}_b
\left( \mathbb{R}
_{y}^{d}\times 
\mathbb{R}_{z}^{d}\right) ^{d+1},$ it results that
\begin{align}
	\left\vert a\left( y,z,w\left(
	y,z\right) ,\mathbf{W}\left( y,z\right) \right) -a\left( y,z,v\left(
	y,z\right) ,\mathbf{V}\left( y,z\right) \right) \right\vert \label{1}\\
	\leq c_{1}\widetilde{\Psi }^{-1}\left( \Phi \left( c_{2}\left\vert w\left(
	y,z\right) -v\left( y,z\right) \right\vert \right) \right) +c_{3}\widetilde{%
		\Phi }^{-1}\left( \Phi \left( c_{4}\left\vert \mathbf{W}\left( y,z\right) -%
	\mathbf{V}\left( y,z\right) \right\vert \right) \right) ;\nonumber\\
	\nonumber\\ 
	\left( a\left( y,z,w\left( y,z\right) ,\mathbf{W}\left( y,z\right) \right)
	-a\left( y,z,w\left( y,z\right) ,\mathbf{V}\left( y,z\right) \right) \right)
	\cdot \left( \mathbf{W}\left( y,z\right) -\mathbf{V}\left( y,z\right) \right)
	\geq 0;\nonumber\\
	\nonumber \\
	\left( a\left( y,z,w\left( y,z\right) ,\mathbf{W}\left( y,z\right) \right)
	\right) \cdot \mathbf{W}\left( y,z\right) \geq \theta \Phi \left( \left\vert 
	\mathbf{W}\left( y,z\right) \right\vert \right),
	\nonumber \\
	\hbox{with }\theta:=\widetilde{\Phi }%
	^{-1}\left( \Phi \left( \underset{t>0}{\min }h\left( t\right) \right)
	\right),\nonumber
\end{align}  for every $z\in 
\mathbb{R}^{d}$ and a.e. $y\in 
\mathbb{R}
^{d},$ where
$a\left( \cdot, \cdot,w,\mathbf{W}\right):=\left( a_{i}\left( \cdot, \cdot,w,\mathbf{W}%
\right) \right) _{1\leq i\leq d}.$ 
\end{corollary}
\begin{proof}
Let $\left( w,\mathbf{W}\right) \in \mathcal{C}_b
\left(\mathbb{R}_{y}^{d}\times 
\mathbb{R}
_{z}^{d}\right)^{d+1}$. From $\left( H_{1}\right)
-\left( H_{4}\right), $ the function $\psi:
\mathbb{R}_{y}^{d}\times 
\mathbb{R}_{z}^{d}\times 
\mathbb{R}_{t}^{d}\times 
\mathbb{R}
_{\zeta }^{d} \to
\mathbb{R}
$ defined by $\psi \left( y,z,t,\zeta \right) :=a_{i}\left( t,\zeta,w(
y,z),\mathbf{W}\left( y,z\right) \right),$
($1\leq i\leq d$),
is an element of  $\mathcal{C}_b\left( 
\mathbb{R}
_{y}^{d}\times 
\mathbb{R}
_{z}^{d};L^{\infty }\left( 
\mathbb{R}
_{t}^{d},\mathcal{C}_b\left( 
\mathbb{R}
_{\zeta }^{d}\right) \right) \right) $ 
%\color{magenta} eventualmente, alla luce di quanto seguira' dopo, scambiare $(y,z)$ con $(t ,\zeta)$. Per ora non cambia nulla, ma magari dopo si'
%\color{black}
and we define its trace ${\widetilde\psi
	^{\varepsilon =1}}$ as above and we get ${\widetilde\psi ^{\varepsilon =1}}\in L^{\infty
}\left( 
\mathbb{R}
_{y}^{d},\mathcal{C}_b\left( 
%TCIMACRO{\U{211d} }%
%BeginExpansion
\mathbb{R}
%EndExpansion
_{z}^{d}\right) \right) .$ Then, the inequalities are immediate consequences of $\left(
H_{1}\right) -\left( H_{4}\right) $ and of Proposition \ref{prop 3.1}.
\end{proof}

In order to give a meaning to $a^{\varepsilon }\left( \cdot, \cdot,w^{\varepsilon },%
\mathbf{W}^{\varepsilon }\right) $ for $\left( w,\mathbf{W}\right) \in 
\mathcal{C}\left( \overline{\Omega };\mathcal{C}_b\left( 
\mathbb{R}_{y}^{d}\times 
\mathbb{R}
_{z}^{d}\right) ^{d+1}\right)$, for every $\varepsilon >0$, the following result can be proven:

\begin{proposition} Let $\varepsilon >0$ and let $a:= (a_i)_{1\leq i \leq d}: \mathbb R^d \times \mathbb R^d \times \mathbb R \times \mathbb R^d \to \mathbb R^d$ satisfy  $(H_1)-(H_4)$.
Let $\left( w,\mathbf{W}\right) \in \mathcal{C}\left( \overline{\Omega };%
\mathcal{C}_b\left( 
\mathbb{R}
_{y}^{d}\times 
\mathbb{R}
_{z}^{d}\right) ^{d+1}\right) $. The function $%
x \in \Omega \rightarrow a_{i}\left( \frac{x}{\varepsilon },\frac{x}{\varepsilon ^{2}}%
,w\left( x,\frac{x}{\varepsilon },\frac{x}{\varepsilon ^{2}}\right) ,\mathbf{%
	W}\left( x,\frac{x}{\varepsilon },\frac{x}{\varepsilon ^{2}}\right) \right) \in
\mathbb{R}
,$ denoted  by $a_{i}\left( \cdot, \cdot, w^{\varepsilon },\mathbf{W}^{\varepsilon }\right)
,$ is well defined and belongs to $L^{\infty }\left( \Omega \right) .$
\end{proposition}

\begin{proof}
Let $x\in \Omega $ fixed. Let $a_{i}\left( \cdot, \cdot,w\left( x,\cdot,\cdot\right) ,\mathbf{%
	W}\left( x,\cdot,\cdot\right) \right) \in L^{\infty }\left( \mathbb{R}_{y}^{d},\mathcal{C}_b\left( 
\mathbb{R}_{z}^{d}\right) \right),$ then, by Corollary \ref{cor42}, the function $\left( y,z\right) \rightarrow
a_{i}\left( y,z,w\left( x,y,z\right) ,\mathbf{W}\left( x,y,z\right) \right) $ is well
defined. Using \eqref{1} we get

\begin{align*}
	\left\Vert a_{i}\left( \cdot, \cdot,w\left( x,\cdot,\cdot\right) ,\mathbf{W}\left(
	x,\cdot,\cdot\right) \right) -a_{i}\left( \cdot, \cdot,w\left( x^{\prime },\cdot,\cdot\right) ,%
	\mathbf{W}\left( x^{\prime },\cdot,\cdot\right) \right) \right\Vert_{L^\infty(
		\mathbb{R}
		_{z}^{d}) }\leq  
	\\ 
	c_{1}\widetilde{\Psi }^{-1}\left( \Phi \left( c_{2}\left\Vert w\left(
	x,\cdot, \cdot\right) -w\left( x',\cdot, \cdot\right) \right\Vert_{\infty }\right)
	\right) + \\ 
	c_{3}\widetilde{\Phi }^{-1}\left( \Phi \left( c_{4}\left\Vert \mathbf{W}%
	\left( x,\cdot, \cdot\right) -\mathbf{W}\left( x',\cdot, \cdot\right) \right\Vert_{\infty }\right) \right),\; x,x' \in \overline{\Omega }.
\end{align*}

\noindent Thus the function $x \in \Omega \rightarrow a_{i}\left( \cdot, \cdot,w\left( x,\cdot, \cdot\right) ,%
\mathbf{W}\left( x,\cdot, \cdot\right) \right) \in$ $\mathcal{C}\left( 
\overline{\Omega };L^{\infty }\left( 
\mathbb{R}_{y}^{d};{\mathcal C}_b(\mathbb R^d_z)\right) \right),$ hence, by Proposition \ref{prop 3.1} the statement follows. 
\end{proof}

\bigskip

\noindent {\bf Acknowledgments}

\noindent The first author acknowledges the support of ICTP-INdAM (2018) and of the University of Salerno which the last author was affiliated with, when
this work started. The last author is a member of INdAM-GNAMPA, whose support is gratefully acknowledged. The authors are indebted with the anonymous referee for his/her comments.

\smallskip

\end{document}